\documentclass[12pt]{article}

\usepackage[a4paper, total={6in, 9in}]{geometry}

\usepackage{mathptmx}
\usepackage{verbatim}
\usepackage{amsfonts}
\usepackage[normalem]{ulem}

\usepackage{amsmath,amsfonts,amssymb,amsthm}
\usepackage{bm}

\usepackage{hyperref}
\usepackage{color}
\usepackage{booktabs}

\usepackage{tikz}
\usepackage{graphicx} 
\usepackage{float}
\tikzset{
    ns/.style={shape=circle, inner sep=1pt, color=black},
    nsf/.style={shape=circle, inner sep=1.2pt, color=black, fill=black} 
}

\usepackage[numbers,sort&compress]{natbib}

\allowdisplaybreaks
\def\qed{\nopagebreak\hfill{\rule{4pt}{7pt}}}

\newtheorem{theorem}{Theorem}[section]

\numberwithin{equation}{section}

\newcommand{\gen}{{\rm Gen}}

\newcommand{\Des}{{\rm Des}}
\newcommand{\Asc}{{\rm Asc}}

\newcommand{\des}{{\rm des}}
\newcommand{\asc}{{\rm asc}}

\newcommand{\inv}{{\rm inv}}
\newcommand{\lrmax}{{\rm lrmax}}
\newcommand{\rlmax}{{\rm rlmax}}
\newcommand{\lrmin}{{\rm lrmin}}
\newcommand{\rlmin}{{\rm rlmin}}
\newcommand{\cyc}{{\rm cyc}}

\newcommand{\exc}{{\rm exc}}
\newcommand{\suc}{{\rm suc}}
\newcommand{\Ima}{{\rm Im}}
\newcommand{\jump}{{\rm jump}}
\newcommand{\drop}{{\rm drop}}
\newcommand{\fix}{{\rm fix}}
\newcommand{\lsuc}{{\rm lsuc}}
\newcommand{\basc}{{\rm basc}}
\newcommand{\lrf}[1]{\lfloor #1\rfloor}

\DeclareMathOperator{\lda}{lda}   
\DeclareMathOperator{\rda}{rda}   
\DeclareMathOperator{\lrda}{lrda} 

\DeclareMathOperator{\ldd}{ldd}   
\DeclareMathOperator{\dd}{dd}     
\DeclareMathOperator{\rdd}{rdd}   

\usepackage{graphicx}
\usepackage{tikz}
\usetikzlibrary{shapes.geometric}
\tikzset{
    every node/.style={circle,   inner sep=1pt,   fill=white!60},  
    tn/.style={shape=circle,   draw,   color=black!70},  
    tn1/.style={shape=circle,   draw},  
    tn2/.style={shape=circle,   draw,   inner sep=1.5pt},  
    tn5/.style={shape=circle,   inner sep=1.6pt,   draw,   color=black!70},  
    tn3/.style={shape=rectangle,   draw,   inner sep=1.5pt},  
    tn4/.style={shape=rectangle,   draw,   inner sep=1.5pt,   color=black!70}, 
    ns/.style={shape=circle,  inner sep=1pt,  color=black}
}

\begin{document}

\noindent

\begin{minipage}{0.5\textwidth}

\raggedright
English Translation\\
The Grammatical Calculus (in Chinese)\\
Scientia Sinica Mathematica\\
56(9) (2026) 1--34.
 \end{minipage}

\begin{center}

 {\Large\bfseries The Grammatical Calculus }
\\[20pt]

William Y.C. Chen

\vskip 3mm

Center for Applied Mathematics\\
Tianjin University\\
Tianjin 300072, P.R. China

\vskip 3mm

Email: { chenyc@tju.edu.cn
}
   
\end{center}

\vspace{0.6cm}

\centerline{\bf Abstract}

 We give an introduction to the grammatical calculus, which originated from the umbral calculus of Rota. Context-free
grammars arise in computer science. Differential operators
can be viewed as the production rules of context-free grammars. By the properties of differential operators, 
we may perform the grammatical calculus on a rigorous
basis,  for the purposes of deriving combinatorial identities and computing generating functions. 
A grammatical labeling serves as a bridge between
a combinatorial structure  and the grammar. 
A grammar may be instrumental in constructing 
a bijection between two combinatorial structures
sharing the same grammar. Such a bijection is called
a grammar assisted bijection. Many combinatorial structures
admit  recursive constructions that are context-free 
in some sense,  such as those corresponding to the Eulerian polynomials,  Bell polynomials, the Ramanujan polynomials and the Narayana polynomials.

\vspace{0.3cm}

\noindent \textbf{Key words:} Grammatical calculus, grammatical labelings, Eulerian polynomials, Bell polynomials, Ramanujan polynomials, Narayana polynomials. 
 
 \noindent \textbf{MSC subject classification:} 05A15,  05A19.

\newpage

\section{Introduction}

 The notion of context-free grammars is due to Chomsky,
 and it forms the foundation of programming languages,
 see, for example, \cite{Chomsky-1,Chomsky-2,Hopcroft-Ullman}. 
Meanwhile, the linear differential operator $D$ 
can be viewed as having a context-free property, namely, 
\begin{equation}
\label{prop:context-free}
     D(fg)  = D(f)\, g + f \, D(g).  
\end{equation} 
In other words, when $D$ is applied to $f$, the result is independent of $g$ and vice versa.  
That is to say that we only
need a differential algebra, that is, an algebra
equipped with a derivation \(D\).
For example, 
consider bivariate polynomials 
in $x$ and $y$, and define the  derivative $D$ by
\begin{equation}\label{gra:Euler}
D(x)= xy,   \; D(y)=xy. 
\end{equation}
From now on, we call such derivatives $D$ formal derivatives.
In the terminology of formal languages,  the operator $D$ defined above
can be represented by the following production rules
\begin{equation}\label{ge}
x \rightarrow xy, \; y \rightarrow xy,
\end{equation}
which turns out to be the grammar for the Eulerian polynomials. 
Strictly speaking,
the grammars we are concerned with  are 
not devised to describe the recursive
generation of a formal language, but rather,
a recursive construction of a combinatorial
structure. 
While the letters of a formal language are
obviously noncommutative, we usually assume that 
 the variables are commutative. The noncommutative
case of the grammatical calculus will be discussed
in the last section. 
 
The idea of 
the grammatical calculus was presented in 
\cite{Chen-1993} along with grammatical derivations
of some classical formulas,
which is regarded as putting the 
traditional symbolic argument on a firm foundation by
Johnson \cite{Johnson}. 
It can be seen that
the Lagrange inversion formula is equivalent to 
Cayley's formula on labeled trees, in the sense that
these two formulas correspond to essentially the same
grammatical formulation, see \cite{Chen-1993}.
It is worth mentioning that 
the umbral calculus provides a 
rigorous mechanism for the classical
symbolic method, see \cite{Rota-1975}. 
In fact, the idea of the grammatical calculus 
was motivated by the umbral calculus of Rota.

The formal derivative \eqref{gra:Euler} can be 
expressed as a differential operator
\begin{align} \label{D-Eulerian}
 D= xy \frac{\partial}{\partial x} + xy 
\frac{\partial}{\partial y}.
\end{align}
Dumont \cite{Dumont-1996} pointed out that
the production rules are more convenient 
in dealing with combinatorial structures.
On the other hand, when it comes to 
computation, it is natural to
employ the above differential operator 
representation while using a CAS (computer 
algebra system), such as Maple and Mathematica. 

 Consider the homogeneous 
 bivariate Eulerian polynomials $A_n(x,y)$,  see \cite{Petersen-2015,  Stanley-I}.   Set $A_0(x,  y)=y$.
 For $n\geq 1$,  define 
 \begin{align}
 \label{eq:A_n}
     A_n(x,  y)= D^n(y).
 \end{align} 
 The first few values of $A_n(x, y)$ are as follows:
\begin{align*}
A_1(x,  y)  &= xy,\\[3pt]
A_2(x,  y) &= xy^2 + x^2y,   \\[3pt]
A_3(x,  y) &= xy^3 + 4x^2y^2 + x^3y,   \\[3pt]
A_4(x,  y) &= xy^4 + 11x^2y^3 + 11x^3y^2 + x^4y,   \\[3pt]
A_5(x,  y) &= xy^5 + 26x^2y^4 + 66x^3y^3 + 26x^4y^2 + x^5y.
\end{align*}
As $A_n(x,  y)$ can be generated by $D$,   
the operator $D$ can be considered a 
creation operator,  but it also possesses the
differential property of an annihilation operator. 
Setting $y=1$, $A_n(x,y)$ becomes the usual Eulerian polynomial $A_n(x)$.

For $n\geq 1$,   let $S_n$ denote the set of
permutations of $[n]=\{1,   2,   \ldots,   n\}$. 
The set of descents and the number of 
descents of a permutation $\sigma=\sigma_1 
\sigma_2 \cdots \sigma_n \in S_n$ are defined by
\[ \Des(\sigma) = \{ i \mid 1\leq i \leq n,  
\sigma_i > \sigma_{i+1} \},   \quad 
\des(\sigma) = \left| \Des(\sigma) \right|,   \]
where we append a zero to the end of \(\sigma\), that is, \(\sigma_{n+1}=0\). Likewise,
we assume that $\sigma_0=0$ 
and define the ascent set
and the number of ascents of $\sigma$ by
 \[ \Asc(\sigma) = \{ i \mid 0\leq i \leq n-1,  
\sigma_i < \sigma_{i+1} \},   \quad 
\asc(\sigma) = \left| \Asc(\sigma) \right|.   \]
An index in the ascent set or descent set is called
an ascent or a descent, respectively.

Below is the combinatorial definition of the
bivariate Eulerian polynomials. 
For $n\ge1$,
\begin{align}\label{defi:Eulerianpol}
A_n(x,  y) = \sum_{\sigma \in S_n} x^{\asc(\sigma)} 
y^{\des(\sigma)}.
\end{align}

The equivalence of the 
operator definition and the combinatorial
definition of $A_n(x,y)$ can be easily
justified by induction. 
However, the idea of a grammatical labeling
just serves this purpose, and it is indeed the
reason it was introduced in \cite{Chen-Fu-2017}. 
Roughly speaking, a grammatical labeling is to associate
a combinatorial structure with labels of a grammar such that
the recursive construction of the
combinatorial structure is consistent with
the production rules of the grammar. In this
way, the grammar replaces the usual recurrence relation in describing the involved combinatorial numbers, see also \cite{Hao-Wang-Yang-2015, 
Ma-2013,  Ma-2013-b,  Zhou-Yeh-Ren-2022,  Zhu-Yeh-Lu-2019}
for more examples of grammatical labelings.

Given the context-free property \eqref{prop:context-free} of a 
grammar, the operator $D$ satisfies the Leibniz formula 
\begin{align}
\label{Leibniz}
    D^n(uv) = \sum_{k=0}^n \binom{n}{k}
    D^k(u)\, D^{n-k}(v). 
\end{align}
In particular,   
\begin{align}
D(u^{-1}) = - u^{-2} \,D(u). 
\end{align}
For example, in the case $D(x)=x$,   we have
\begin{align}
\label{x-1}
D(x^{-1}) = - x^{-1}. 
\end{align}

Via examples, we will 
give an outline of the grammatical calculus, involving
set partitions, the Bell numbers,  the
Eulerian polynomials, the peak polynomials of permutations, 
up-down run polynomials,  
the Diaconis-Evans-Graham theorem, 
the operator expansion of
$(g(x)D)^nf(x)$,  the Ramanujan polynomials, 
plane trees, the ballot numbers and the Narayana numbers.

\section{The Umbral Calculus and the Grammatical Calculus}

The simplest example of the umbral calculus is
the following argument of Rota for the binomial inversion:
The relation
\begin{align} \label{inversion-1}
    a_n = \sum_{k=0}^n \binom{n}{k} b_k 
\end{align} 
holds for all $n$ if and only if 
\begin{align} \label{inversion-2}
b_n = \sum_{k=0}^n (-1)^{n-k}
\binom{n}{k} a_k
\end{align}
holds for all $n$.
The idea of the classical symbolic 
method is to treat $a_n$ as $a^n$. 
In doing so, (\ref{inversion-1}) 
can be written as $a^n=(b+1)^n$. 
On the other hand,  $b_n$ can be written as
$(b+1-1)^n$.
Thus (\ref{inversion-2}) follows from the
binomial theorem.
Rota \cite{Rota-1975} defined the linear
functional
$L(x^n)=b_n$, so that (\ref{inversion-1}) takes the form
$a_n=L((x+1)^n)$. Now we have
$b_n=L(((x+1)-1)^n)$, and so the binomial expansion
yields (\ref{inversion-2}). 
This explains why one can get valid results while
treating $a_n$ as $a^n$. 

Let us take the above example to illustrate the
idea of the grammatical calculus. Define
\[ x \rightarrow x,  \quad 
b_i \rightarrow b_{i+1} \;(i \geq 0). \]
Let $D$ be the formal derivative of the above grammar; i.e., 
$$
D(x)=x, \; D^i(b_0)=b_i \; (i \ge 0).
$$
By the Leibniz formula \eqref{Leibniz},
  (\ref{inversion-1}) can be expressed as
 $a_n=x^{-1}D^n(b_0x)$. Note that
\begin{align*}
    b_n=D^n(b_0) =D^n(b_0x x^{-1}).
\end{align*}
Since $D(x^{-1})=-x^{-1}$, applying the Leibniz
formula gives  (\ref{inversion-2}).

The second example is concerned with the following identity 
\begin{align} \label{df}
    (2n)!! = \sum_{k=0}^n \binom{n}{k}
    (2k-1)!!\,  (2(n-k)-1)!!,  
\end{align}
where $(2k-1)!!= 1\cdot 3 \cdot  \cdots \cdot
(2k-1)$,   $(2n)!!= 2\cdot 4 \cdot \cdots
\cdot (2n)$,     $(-1)!!$ and $0!!$
are defined to be 1,  
see Triana-Castro  \cite{Triana-2019}.

Let $D$ be the formal
derivative of the grammar $a\rightarrow a^3$.
For $n\geq 0$, it is clear that
\[ D^n(a) = (2n-1)!! \,  a^{2n+1},   \quad 
   D^n(a^2) = (2n)!!\,   a^{2n+2}. \]
Expanding $D^n(a^2)$ using 
the Leibniz formula, we obtain (\ref{df}), 
which is equivalent to the classical identity:
\[ \sum_{k=0}^n \binom{2k}{k} \binom{2n-2k}{n-k} =4^n. \]

The third example is the formula for
the higher order derivatives 
of a composite function $f(g(x))$,  that is, the 
Fa\`a di Bruno formula. 
Let $f_i=f^{(i)}(g(x))$ and $g_i=g^{(i)}(x)$. Then we have
\begin{align}
    (f(g(x)))^{(n)}= \sum_{k=0}^n f_k 
    \sum_{k_1,   \,   k_2,   \,   \ldots,   \,   k_n \ge 0}
    \frac{n!}{k_1! k_2! \cdots k_n! \,   1!^{k_1}
    2!^{k_2} \cdots n! ^{k_n}} g_1^{k_1} g_2^{k_2} \cdots g_n^{k_n},  
     \end{align}
     where the second sum ranges over
     $k_1+k_2+\cdots+k_n=k$ and $k_1+2k_2+\cdots + nk_n=n$.
The coefficients in the above expansion
have a combinatorial interpretation. A partition 
of a set $[n]$ can be written as
$\{B_1,B_2,\dots,B_k\}$, where the $B_i$ are disjoint nonempty
subsets of $[n]$ such that $B_1 \cup \dots \cup B_k = [n]$,
and each $B_i$ is called a block. 
The type of a partition $\pi$ is denoted by
$1^{k_1}2^{k_2}\cdots n ^{k_n}$, where $k_i$
is the number of blocks of $\pi$ with 
$i$ elements.  The coefficient in the Fa\`a di Bruno
formula equals the number of partitions of $[n]$ having type $1^{k_1}2^{k_2}\cdots n ^{k_n}$. 
The classical proof as given in
\cite{Riordan-1958} is a paradigm of the
symbolic method. Below is an argument
using the grammar
\begin{align}
\label{gram-bruno}
    f_i \rightarrow f_{i+1}\,g_1 \;( i \geq 0) ,   \quad  g_i 
\rightarrow g_{i+1} \; ( i \geq 1).
\end{align} 
Here are some computations associated with this grammar:
\begin{align*}
    D(f_0) &= f_1g_1, \\[3pt]
    D^2(f_0) &  = D(f_1)\,  g_1 + f_1 \,   D(g_1) = 
    f_2\,   g_1^2 + f_1\,    g_2,   \\[3pt]
   D^3(f_0)  &= f_3 \,    g_1^3 + 3 \,   f_2 \,   
    g_1 \,    g_2
    + f_1 \,    g_3.
\end{align*}
 In general, 
  $D^n(f_0)$ agrees with the
  formula of Fa\`a di Bruno. We now explain this formula by means of a grammatical labeling.
  Given a partition $\sigma$
  of $[n]$ with $k$ blocks, 
    we use $f_k$ to label the whole
    partition and use $g_i$ to label a block
    with $i$ elements, so that the
    weight of $\sigma$ equals
$f_k g_1^{k_1} g_2^{k_2} \cdots g_n^{k_n}$,
where $k_i$ denotes the number of 
blocks with $i$ elements. The grammar
 \eqref{gram-bruno} reflects the 
 two cases of constructing a partition
 of $[n+1]$ from a partition of 
$[n]$: (1) Insert $n+1$ as a new block.
This operation is captured by the production rule
$f_k\rightarrow f_{k+1}g_1$. (2) Insert 
$n+1$ into a block containing $i$ 
elements. In this case, we have
the production rule $g_i\rightarrow
g_{i+1}$. 

Using a simplified version of the above 
grammar, that is,  \begin{align}\label{ax}
   G= \{a \rightarrow a x,  \; x \rightarrow x\}, 
\end{align}  
we can
generate the Stirling numbers of the second kind  
and the Bell numbers.
To be more specific,
for $1 \leq k \leq n$,  the Stirling number
of the second kind $S(n,  k)$ refers to the number of partitions
of $[n]$ with $k$ blocks, and  the Bell number
$B_n$ refers to the number
of partitions of $[n]$.

\begin{theorem}[Chen \cite{Chen-1993}]
Let $D$ be the formal derivative of the
grammar (\ref{ax}). For $n\geq 1$, we have
\begin{align} \label{S}
    D^n(a)& = a \sum_{k=1}^n S(n,  k)\, x^k,   \\[6pt]
    B_n  & = a^{-1} D^n(a) \big| _{ x=1}. \label{B}
\end{align}
\end{theorem}

In the following two examples, we illustrate
how the grammatical calculus can be employed to
prove identities involving the Bell numbers.
Spivey \cite{Spivey-2008} obtained the following
identity
\[ B_{n+m}= \sum_{k=0}^n \sum_{j=0}^m
 j^{n-k} S(m,  j)\binom{n}{k}B_k. \]
Using the grammar (\ref{ax}),   we see that for $i\geq 0$,  
$D^i (x^j)= j^i  x^j$. Thus, 
\[ D^{n+m}(a) = D^n(D^m(a)) = D^n
\left( \sum_{j=0}^m S(m,  j) ax^j \right) = 
\sum_{j=0}^m S(m,  j)D^n(ax^j).\] 
Applying  the Leibniz formula to $D^n(ax^j)$ and
setting $a=x=1$, we arrive at 
Spivey's identity. 
A refinement of this identity goes back to  Verde-Star \cite{Verde-Star-1988}, and a grammatical derivation
can be found in \cite{Chen-1993}. 
For generalizations and applications of this
relation, see  \cite{Gould-2008, Shattuck-2016, Xu-2012}.

 Shattuck \cite{Shattuck-2015} obtained the following
identity involving the Bell numbers. For $0 \leq j \leq n$,    
\begin{align}
\label{eq:Bell}
\sum_{i=0}^{j}(-1)^{i}\binom{j}{i} B_{n+1-i}=\sum_{k=0}^{n-j}\binom{n-j}{k} B_{n-k}.  
\end{align}
Using the grammar
  (\ref{ax}), according to (\ref{x-1}),  the sum
\[ \sum_{i=0}^{j}(-1)^{i} x^{-1}\binom{j}{i} D^{n+1-i}(a), \]
can be expressed as
\[ \sum_{i=0}^{j}\binom{j}{i} D^{i}\left(x^{-1}\right) D^{j-i}\left(D^{n+1-j}(a)\right) . \]
By the Leibniz formula \eqref{Leibniz}, this equals
\[ D^{j}\left(x^{-1} D^{n+1-j}(a)\right). \]
Since $D(a)=a x$ and $D(x)=x$, we further get
\begin{align*}
D^{j}\left(x^{-1} D^{n-j}(a x)\right) & =D^{j}\left(x^{-1} \sum_{k=0}^{n-j}\binom{n-j}{k} D^{k}(x) \,D^{n-j-k}(a)\right) \\[6pt]
& =\sum_{k=0}^{n-j}\binom{n-j}{k} D^{n-k}(a).
\end{align*}
Hence we obtain
\begin{equation}\label{spivey}
\sum_{i=0}^{j}(-1)^{i} x^{-1}\binom{j}{i} D^{n+1-i}(a)=\sum_{k=0}^{n-j}\binom{n-j}{k} D^{n-k}(a).
\end{equation}
Setting $a=x=1$, we are led to \eqref{eq:Bell}. 

For $n\geq 1$, the exponential polynomials are defined by
\[ \phi_n(x) = \sum_{k=1}^n S(n,  k)\,x^k,\]
which permits a grammatical representation
\[ \phi_n(x) = a^{-1}D^n(a). \]
Therefore, \eqref{spivey} can be written as
\begin{equation}\label{eq:Exp}
\sum_{i=0}^j (-1)^i \binom{j}{i} \phi_{n+1-i}(x) = x \sum_{k=0}^{n-j}
\binom{n-j}{k} \phi_{n-k}(x). 
\end{equation}
The involution for (\ref{eq:Bell})
given in \cite{Shattuck-2015} also applies to the 
  relation (\ref{eq:Exp})
  for the exponential polynomials.

\section{Grammatical Labelings}

In this section, we give three examples
to demonstrate the notion of a grammatical
labeling. In fact, we have already encountered 
a grammatical labeling of partitions related to
the Fa\`a di Bruno formula. 
Roughly speaking, 
a grammatical labeling is introduced to facilitate
the conversion of a recursive
construction of a combinatorial structure into a 
context-free grammar, rather than using recurrence relations. 
Then this grammar itself is sufficient to
perform the grammatical calculus 
for the purpose of enumeration.

The first example is a grammatical
labeling for the cycle representation of a permutation.
Let $n \geq 1$ and $\sigma$ be a permutation of $[n]$. 
For 
$n\geq 2$, it is known that
the number of permutations of $[n]$ such that
$1$ and $2$ belong to the same cycle
equals the number of permutations such that
$1$ and $2$ occur in different cycles, see
 \cite{Pozdnyakov-Steele}. In the context of a
 grammatical labeling, this fact becomes obvious.
 Assume that we always place the minimum element of 
 a cycle at the beginning, and label the position 
 after each element by $y$, which indicates this
 is a valid position to insert a new element in the
 recursive construction of a permutation  of $[n+1]$ 
 (in the cycle
 notation) from a permutation  of $[n]$ (in the cycle
 notation).
 In addition, we place a label $a$ at the end outside the cycles to signify a position to
 insert $n+1$ as a new cycle. For instance, here is the labeling
 of  the cyclic representation of a permutation:
\begin{align*}
    (1\, y \,7 \,y \,3)\;(2\, y \,4\,y)\;(5\,y\, 6\,y)\;a.
\end{align*}
 Then the recursive construction of permutations yields the
 following grammar:
\begin{align}
\label{gram:cycle}
a\rightarrow a y, \quad y \rightarrow y^2 .
\end{align} 
When $n\geq 2$, 
the total weight of permutations such that
$1$ and $2$ belong to the same cycle 
is given by $D^{n-2}(a y^2)$. Likewise,
the total weight of permutations such that
$1$ and $2$ occur in different cycles is also given by
$D^{n-2}(a y^2)$, as expected.

If we take the number of cycles into account,
we may use $x$ to label a cycle. Consequently,
we come to the grammar
\begin{align} 
 \label{axy}   
a\rightarrow a  x  y, \quad y \rightarrow y^2 .
\end{align}
Let $c(n,k)$ denote the signless Stirling number 
of the first kind, i.e., the
number of permutations of $[n]$ with
$k$ cycles,  and let $D$ be the formal 
derivative of the above grammar (\ref{axy}).
Then for $n\geq 1$, 
\begin{align} 
D^n(a) =a y^n\sum_{k=1}^n c(n,k)\,  x^k. 
\end{align}
   Ma \cite{Ma-2013-b} found the following
   grammar for $c(n,k)$,
\begin{align}
\label{grammar-cnk-2}
a\rightarrow ab, \quad
b \rightarrow bc, \quad c \rightarrow c^2.
\end{align}
It has been shown that for $n \ge 1$, 
\begin{align*}
    D^n(a) =a \sum_{k=1}^n c(n,k) b^kc^{n-k}.
\end{align*}
Setting $b=xy$, $c=y$, 
the grammar (\ref{grammar-cnk-2}) gives rise
to the grammar (\ref{axy}).

There is a simple operation of breaking cycles 
which explains the equality
between the number of permutations
with $1$ and $2$ in the same cycle and the
number of permutations with $1$ and $2$ in
different cycles. 
This operation has been employed by Chen \cite{Chen-2024}
to give a combinatorial interpretation 
of the fact that for $n\geq 1$,
the number of permutations of 
$[2n]$ all of whose cycles have odd length equals the
number of permutations of $[2n]$ all of whose cycles have even length. Using the same grammatical 
reasoning, we may deduce that for $n\geq r \geq 2$, 
 the number of 
 permutations of $[n]$ with $1$, $2$, $\ldots$,  $r$
 in the same cycle equals $(r-1)!$ times 
 the number of permutations of $[n]$ with
 $1$, $2$, $\dots$,   $r$ in different cycles. These two numbers are \(n!/r\) and \(n!/r!\), respectively.
 This number can be obtained via a 
 grammatical argument. 
 As for the grammar \eqref{gram:cycle}, we see that
 \[ D(ay^i)=(i+1)ay^{i+1},\]
 and so
 \[D^{n-r}(ay^r) \mid_{a=y=1}=(r+1)(r+2)\cdots n.\] 
In view of the grammar (\ref{axy}), we may
treat $x$ as a constant. Thus, for $i\geq 0$, we have
\[ D(ay^i)=(x+i)ay^{i+1}.\]
It follows that for $n\geq 1$,  
\[ \sum_{k=1}^n c(n,k) x^k = x(x+1) \cdots (x+n-1). \]

The second example is a grammatical labeling
for the bivariate Eulerian polynomials \eqref{eq:A_n}, which can be defined by a differential operator or a grammar.
Assume that $n\geq 1$. For a permutation
$\sigma=
\sigma_1\sigma_2\cdots \sigma_n$ of $[n]$, 
we adopt the convention that
$\sigma_0=\sigma_{n+1}=0$.  For $0 \leq i \leq n-1$,  
if $i$ is an ascent, that is, $\sigma_i< \sigma_{i+1}$, 
then place a label $x$ between $\sigma_i$ and $\sigma_{i+1}$;
otherwise, place  a label $y$ instead. When
the element $n+1$ is inserted at the
position between $\sigma_{i}$  and $\sigma_{i+1}$, the
change of labels is captured by the production
rules $x\rightarrow xy$ and $ y \rightarrow xy$. 

For example,   let $n = 9$ and $\sigma = 8\; 4\; 9\; 6\; 1\; 2\; 5\; 3\; 7$. The labeling of $\sigma$ is shown below: 
\begin{align}
\label{x8}
x\;8\; y\;4 \;x\;9\;y\; 6\;y\; 1\;x\; 2 \;x\;5 \;y\;3 \; x\;7 \; y.
\end{align}

The bivariate Eulerian polynomials
can also be defined on increasing
binary trees. For the correspondence
between permutations and increasing binary trees, 
see Stanley \cite{Stanley-I}. 

To assign labels to increasing binary trees, we need
to adopt the version of complete increasing binary trees.
To be more specific, a complete
increasing binary tree is generated by
attaching leaves to make the increasing binary tree
complete.  The internal vertices form an increasing
binary tree, and a left leaf is labeled by $x$
and a right leaf is labeled by $y$. The labels are
assigned only to the leaves, see Figure \ref{fig1}.  

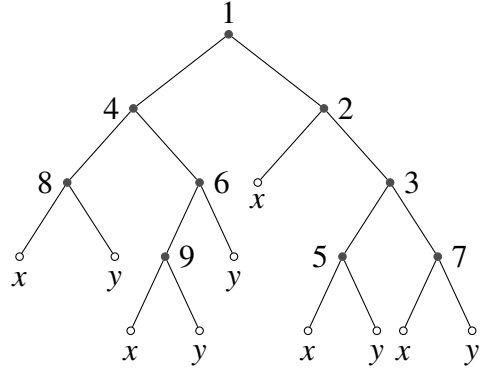
\begin{figure}[!ht]
\begin{center}
\begin{tikzpicture}[scale=0.7]
\node [tn,  label=90:$1$]{}[grow=down]
	[sibling distance=36mm,  level distance=14mm]
    child {node [tn,  label=180:{$4$}](four){}
       [sibling distance=25mm,  level distance=14mm]
    child {node [tn,  label=180:{$8$}](eight){}
     [sibling distance=18mm,  level distance=14mm]
     child {node [tn1,  label=-90:{}](eightl){}}
     child {node [tn1,  label=0:{}](eightr){}}
     }
      child {node [tn,  label=0:{$6$}](six){}
     [sibling distance=13mm,  level distance=14mm]
     child {node [tn,  label=0:{$9$}](nine){}
     child {node [tn1,   label=-90:{}](ninel){}}
     child {node [tn1,   label=-90:{}](niner){}}
     }
     child {node [tn1,  label=0:{}](sixr){}}
     }
     }
     child {node [tn,  label=0:{$2$}](twoa){}
     [sibling distance=25mm,  level distance=14mm]
      child {node [tn1,  label=-90:{$x$}](twoal){}}
    child {node [tn,  label=0:{$3$}](threea){}
       [sibling distance=18mm,  level distance=14mm]
    child {node [tn,  label=180:{$5$}](five){}
       [sibling distance=13mm,  level distance=14mm]
      child {node [tn1,  label=-90:{}](fivel){}}
    child {node [tn1,  label=-90:{}](fiver){}}
    }
    child {node [tn,  label=0:{$7$}](sevena){}
         [sibling distance=13mm,  level distance=14mm]
    child {node [tn1,  label=-90:{}](sevenl){}}
    child {node [tn1,  label=-90:{$y$}](sevenr){}}
    }
     }};
    \node [below=3pt] at (eightl){$x$}; 
     \node [below=3pt] at (eightr){$y$}; 
     \node [below=3pt] at (sixr){$y$};
     \node [below=3pt] at (ninel){$x$};
     \node [below=3pt] at (niner){$y$};
     \node [below=3pt] at (fivel){$x$}; 
     \node [below=3pt] at (fiver){$y$};
    \node [below=3pt] at (sevenl){$x$};
    \node [below=3pt] at (sevenr){$y$};
\end{tikzpicture}
\end{center}
\caption{A complete increasing binary tree with \((x,y)\)-labels.}
\label{fig1}
\end{figure}

The third example is connected with $m$-regular partitions.  
A partition of $[n]$ is said to be $m$-regular
if any two elements in the same block differ by at least
$m$, that is, the absolute difference of any two elements
in the same block is at least $m$. 
This notion arises in computational biology. In particular, the number of \(2\)-regular partitions of
\([n]\) is the Bell number \(B_{n-1}\), where \(B_{n}\) is 
the number of partitions of $[n]$.
Yang \cite{Yang} showed that
the number of $m$-regular partitions of
$[n]$ equals the number of $(m-1)$-regular partitions
of $[n-1]$.   Chen-Deng-Du \cite{Chen-Deng-Du-2005} presented a bijective algorithm to reduce an $m$-regular partition
to an $(m-1)$-regular partition and showed that this reduction
preserves the noncrossing property.

From the perspective of grammatical calculus,
the generation of $m$-regular partitions and partitions without restrictions are governed by the
same grammar.
First, notice that in an $m$-regular partition of $[n]$, where
$n \geq m$, the largest $m-1$ elements 
$n-m+2,   n-m+3,   \ldots,   n$ are scattered in
different blocks. These  
$m-1$ blocks are not labeled, whereas 
the remaining blocks are labeled by $x$.
 Now, we cannot insert $n+1$ into any 
 unlabeled block, but we can insert $n+1$ into 
 any block labeled by $x$. In addition, we
 label the whole partition by $a$.  
Let us examine how the labels change
after the insertion of the
element $n+1$. 

There are two cases. Case 1. When $n+1$ forms a new block, the 
block containing $n-m+2$ gets a label $x$,  since 
the insertion of  $n+2$  
into this block still yields an $m$-regular partition. Consequently, the new block consisting of $n+1$ has no label. 
Case 2. When $n+1$ is inserted into a block
labeled by $x$, the label \(x\) is removed from the resulting block, but the block containing $n-m+2$
will get a label $x$. This labeling
scheme gives a recursive construction
of $m$-regular partitions, along with a grammar
\begin{align}\label{axm}
    a \rightarrow a x,   \; x \rightarrow x. 
\end{align}   
Indeed, this is just the grammar to generate
ordinary partitions.

\begin{theorem}
Assume that $m\geq 1$ and  $n\geq m$. Let
$D$ be the formal derivative of the grammar (\ref{axm}). 
Then the total weight of $m$-regular 
partitions of $[n]$ equals $D^{n-m}(ax)$. 
\end{theorem}

Comparing the procedures to generate
$m$-regular partitions and $(m-1)$-regular partitions,   
we are led to a bijection, called a 
grammar assisted bijection,
which turns out to be  equivalent  to the 
reduction algorithm of 
Chen-Deng-Du \cite{Chen-Deng-Du-2005}. 
It appears that the grammar setting of the bijection 
is easier to justify.

\section{Eulerian Polynomials and the Grammatical Calculus}

The definition of the bivariate Eulerian
polynomials was given in \eqref{defi:Eulerianpol} in the Introduction. Setting 
$y=1$, $A_n(x,y)$ becomes the 
Eulerian polynomial $A_n(x)$. Eulerian polynomials
are the generating functions of permutations 
with respect to the
numbers of ascents (or descents); see the labeling in the example \eqref{x8}. 
It should be noted that 
they differ from the Euler numbers and the
Euler polynomials. 

There is a classical theorem of
Foata-Sch\"utzenberger \cite{Foata70} on the
Eulerian polynomials:  
For $n \geq 1$,    $A_n(x)$ can be uniquely 
expanded as 
\begin{align}
A_n(x) = \sum_{k=1}^{\lfloor (n+1)/2 \rfloor} \gamma_{n,  k} \,x^k (1+x)^{n-2k+1}, 
\label{A_n_Gamma}
\end{align}
where $\gamma_{n,  k}$ are nonnegative integers. 
The expression  (\ref{A_n_Gamma}) is called the 
 $\gamma$-expansion of $A_n(x)$. It can be reformulated in
 the bivariate form:
\begin{align}
A_n(x,  y) = \sum_{k=1}^{\lfloor (n+1)/2 \rfloor} \gamma_{n,  k}\, (xy)^k (x+y)^{n-2k+1}.
\end{align}
The coefficients
$\gamma_{n,  k}$ are called the $\gamma$-coefficients.
The nonnegativity of these coefficients
is called the $\gamma$-positivity. 

Ma-Ma-Yeh \cite{Ma-Ma-Yeh} 
gave a grammatical derivation of the
$\gamma$-positivity of $A_n(x,  y)$. 
Given the grammar (\ref{ge}), namely,
$G= \{x \to xy,\;  y\to xy\}$, 
make the following change of variables
$$u = xy,   \quad v = x + y.$$
Clearly,
\begin{align*}
D(u) &= D(xy) = xy^2 + x^2y = uv,  \\
D(v) &= D(x + y) = 2xy = 2u.  
\end{align*}
Thus $G$ takes the following form in the variables
$u$ and $v$: 
$$H = \{u \to uv,   v \to 2u\}.$$
For $n\ge1$,   we have
\[ A_n(x,  y)= D_G^n(y)=D_H^{n-1}(u) \mid_{u=xy,\, v=x+y},\] 
where $D_G$ and $D_H$ are the formal
derivatives of $G$ and $H$, respectively.
Now the $\gamma$-positivity of $A_n(x,y)$ becomes
transparent. Moreover, the known combinatorial
interpretations  of the $\gamma$-coefficients
in terms of permutations and binary trees can also
be deduced from the grammar.

Euler obtained the following
identity for $A_n(x)$, see \cite{Hyatt-2016}. Here we 
give a derivation using the grammatical calculus, see
Chen-Fu \cite{Chen-Fu-2017}. 
 
\begin{theorem}  
For $n \ge 1$,   
\begin{align}
A_n(x) = \sum_{k=0}^{n-1} \binom{n}{k} A_k(x)(x-1)^{n-1-k},   \label{Eq:Eulerian}
\end{align}
where $A_0(x) = x$. 
\end{theorem}

\begin{proof}
For the grammar (\ref{ge}),  we have 
$D(x^{-1}) = -x^{-2}D(x) = -x^{-1}y$, and
\begin{align}
D(x^{-1}y) = x^{-1}D(y) + yD(x^{-1}) = x^{-1}y(x-y). \label{2.2}
\end{align}
Since \(D(x-y)=0\),  we find that
\begin{align}
D^n(x^{-1}y) = x^{-1}y(x-y)^n. \label{2.3}
\end{align}
By the Leibniz formula, for $n \ge 1$,  
\begin{align}
\begin{aligned}
D^n(y) = D^n(xx^{-1}y) &= \sum_{k=0}^n \binom{n}{k} D^k(x)\,D^{n-k}(x^{-1}y) \\
&=\sum_{k=0}^{n-1} \binom{n}{k} D^k(x)\,D^{n-k}(x^{-1}y)+x^{-1}y\, D^n(x). \label{2.4}
\end{aligned}
\end{align}
Plugging (\ref{2.3}) into (\ref{2.4}),  we get
\begin{align*}
(x-y)x^{-1}\,D^n(y) = \sum_{k=0}^{n-1} \binom{n}{k} x^{-1}y\,D^k(x)\,(x-y)^{n-k}.
\end{align*}
Here we have used the fact that for $n\ge 1$, $D^n(x)=D^n(y)$. 
Setting $y = 1$ yields (\ref{Eq:Eulerian}). 
\end{proof}

As remarked by Carlitz-Scoville \cite{Carlitz-Scoville-1974},
it is not so easy to 
recover the generating function
of the Eulerian polynomials 
from the recurrence relation of the
Eulerian numbers. 
The grammatical calculus for this purpose
seems to be a shortcut, see Chen-Fu \cite{Chen-Fu-2022}.

\begin{theorem}  \label{thm-a-n}
 Let $A_0(x,  y)=y$. We have
\begin{align} \label{g-a-x-y}
\sum_{n=0}^\infty A_n(x,  y) \frac{t^n}{n!} = \frac{y-x}{1 - xy^{-1} e^{(y-x)t}}.
\end{align}
\end{theorem}

 Below is a grammatical proof of (\ref{g-a-x-y}). 
Let $f$ be a 
Laurent polynomial in $x$ and $y$, 
and $D$ be the formal derivative of
the grammar \eqref{ge}. Define the exponential
generating function of $f$ with respect to $D$ by
\begin{align} \label{Gen-d}
\gen(f,   t) = \sum_{n \geq 0} D^n(f) \frac{t^n}{n!}.
\end{align}
Let $g$ also be a Laurent polynomial in $x$ and $y$. 
By the linearity of $D$ and
the Leibniz formula, the following multiplicative
property holds:
\begin{align} \label{Gen-fg}
\gen(fg,   t) = \gen(f,   t)\, \gen(g,   t).
\end{align}

\begin{proof}
Under the convention $A_0(x,  y)=y$,  for $n \ge 0$,  
we have $A_n(x,   y)=D^n(y)$, 
where $D$ is the formal derivative of  the
grammar  (\ref{ge}).
We aim to 
compute $\gen(y,   t)$.
By the multiplicative property (\ref{Gen-fg}), 
we have
\[ \gen(y, t)=\frac{1}{\gen(y^{-1},t)}. \]
To compute $\gen(y^{-1},t)$, we find that
\begin{align} \label{D-x-1}
D(y^{-1} ) = - y^{-2} D(y) = -x y^{-1}. 
\end{align}
Thus, for  $n \geq 0$,
\begin{align} 
D^n(xy^{-1} ) = x y^{-1} (y-x)^n.\label{D-x-1-y}
\end{align}
 By (\ref{D-x-1}),  we get
\begin{align*}
\gen(y^{-1},  t) = \sum_{n \geq 0} D^n(y^{-1}) \frac{t^n}{n!} = y^{-1} - \sum_{n \geq 1} D^{n-1} (xy^{-1}) \frac{t^n}{n!}.
\end{align*}
Employing (\ref{D-x-1-y}) and the constant property
of $y-x$, we obtain
\begin{align*}
\gen (y^{-1},  t) &= y^{-1} - \sum_{n \geq 1} xy^{-1}(y-x)^{n-1} \frac{t^n}{n!} \\
&= y^{-1} - \frac{xy^{-1}}{y-x} \left( e^{(y-x)t} - 1 \right) \\
&= \frac{1 - xy^{-1} e^{(y-x)t}}{y-x},  
\end{align*}
and thus,
\begin{align} 
\gen(y,   t) =  \frac{y-x}{1 - xy^{-1} e^{(y-x)t}},
\end{align} 
completing the proof. \qed
\end{proof}

Next we consider refinements of Eulerian polynomials. 
A permutation statistic is called Eulerian if it 
is equidistributed with respect to the excedance number \(\exc\) over \(S_n\) for every \(n\). Let \(\sigma = \sigma_1\sigma_2\cdots\sigma_n \in S_n\)
be a permutation of $[n]$, an index $1\le i \le n$ 
such that \(\sigma_i > i\) is called an excedance of 
\(\sigma\) and the number of excedances of 
$\sigma$ is denoted by  \(\exc(\sigma)\).
A Stirling statistic of a permutation 
refers to a statistic that is
equidistributed with respect to the number of cycles.

Examples of Stirling statistics include the
number of left-to-right maxima, the number of 
right-to-left maxima. Let  $\cyc(\sigma)$
denote the number of cycles of $\sigma$.  
A left-to-right maximum (minimum) refers
to an element $\sigma_i$ such that 
for all $j<i$, $\sigma_j < \sigma_i$ ($\sigma_i < \sigma_j$).
Likewise, a right-to-left
maximum (minimum) of $\sigma$ refers to an element  $\sigma_i$
such that for all $j > i$ we have $\sigma_j < \sigma_i$ ($\sigma_j > \sigma_i$).  
Let $\lrmax(\sigma)$, $\lrmin(\sigma)$, $\rlmax(\sigma)$ and  $\rlmin(\sigma)$ denote the numbers of 
left-to-right maxima,
left-to-right minima,
right-to-left maxima,
and right-to-left minima, respectively.   

Foata-Sch\"utzenberger \cite{Foata70} introduced the
following polynomials, involving the
Eulerian-Stirling statistics, which
have been called
the $q$-Eulerian polynomials by Brenti \cite{Brenti-2000}: 
\begin{align}\label{brenti}
A_n(x;q) = \sum_{\sigma \in S_n} x^{\exc(\sigma)} q^{\mathrm{cyc}(\sigma)}. 
\end{align}

Brenti \cite{Brenti-2000} obtained the
generating function of $A_n(x;q)$,  and showed 
that $A_n(x; q)$ is log-concave. 
A grammar to generate 
$A_n(x; q)$ has been given by 
Ma-Ma-Yeh-Zhu \cite{Ma-Ma-Yeh-Zhu}. 

Carlitz-Scoville \cite{Carlitz-Scoville-1974} introduced
the polynomials $A_n(x,   y \mid \alpha,   \beta)$, which
have been called   the  $(\alpha,   \beta)$-Eulerian 
polynomials by Ji \cite{Ji-2025}. 
They are defined by the numbers of descents, ascents,
left-to-right maxima, right-to-left maxima.
More precisely,  $A_0(x,  y\mid \alpha,  \beta) = 1$ and 
for $n\geq 1$,
\begin{align}\label{eq:2.1}
A_n(x,  y\mid \alpha,  \beta)=\sum_{\sigma\in S_{n+1}}x^{\asc(\sigma)-1}y^{\des(\sigma)-1}\alpha^{\lrmax(\sigma)-1}\beta^{\rlmax(\sigma)-1}.  
\end{align}

  By the first fundamental transformation of Foata-Sch\"utzenberger~\cite{Foata70},  it can be seen 
  that the $q$-Eulerian polynomials \eqref{brenti} 
  are a specialization of $A_n(x,   y \mid  \alpha,   \beta)$, 
  as given by
\begin{align*}
A_n(x;q) = A_n(1,   x \mid q,   0).
\end{align*}

 Setting $\alpha=0$ and  $\beta = 1$ in (\ref{eq:2.1}), we recover the bivariate Eulerian polynomials, that is, for
 $n \ge 1$, 
\begin{align}
    A_n(x,  y\mid 0, 1) = \sum_{\sigma\in S_{n} } x^{\asc(\sigma)-1}y^{\des(\sigma)}=x^{-1}A_{n}(x,  y).
\end{align}

Below are the first few values of 
\(A_n(x,  y\mid \alpha,  \beta)\): 
\begin{align*}
A_1(x, y \mid \alpha, \beta) &= \alpha x + \beta y, \\
A_2(x, y \mid \alpha, \beta) &= \alpha^2 x^2  + \alpha xy + \beta xy +    2\alpha\beta xy +  \beta^2 y^2 \\
A_3(x, y \mid \alpha, \beta) &=    \alpha^3 x^3   
 + \alpha x^2 y  +\beta x^2 y  + 3\alpha \beta x^2 y + 
 3\alpha^2 x^2 y  
+   3\alpha^2 \beta x^2 y   \\
&\qquad +{} 
 \alpha xy^2 +\beta xy^2 +
3\alpha\beta  xy^2 + 3\alpha\beta^2 xy^2 + 3\beta^2 xy^2 + \beta^3 y^3.
\end{align*}

Carlitz-Scoville \cite{Carlitz-Scoville-1974}
obtained the generating function of 
\(A_n(x,  y \mid \alpha,  \beta)\). 

\begin{theorem}[Carlitz-Scoville \cite{Carlitz-Scoville-1974}] \label{thm:carbi}
We have
\begin{align}\label{eq:1.7}
\sum_{n\ge 0} A_n(x,  y \mid \alpha,  \beta)\frac{t^n}{n!}
= \bigl(1+xF(x,  y;t)\bigr)^\alpha \bigl(1+yF(x,  y;t)\bigr)^\beta,  
\end{align}
where  
\begin{align}\label{eq:1.8}
F(x,  y;t) = \frac{e^{xt}-e^{yt}}{x e^{yt}-y e^{xt}}.
\end{align}
\end{theorem}

The complement of a permutation
$\sigma=\sigma_1\sigma_2 \cdots \sigma_n$ is defined as 
\begin{align} 
\label{def:com}
\sigma^c = (n+1 - \sigma_1)\,
(n+1 - \sigma_2) \, \cdots \, (n+1 - \sigma_n).
\end{align}
Taking the complementation of permutations and exchanging the roles of \(x\) and \(y\),  
Ji \cite{Ji-2025} obtained an equivalent definition 
\begin{align}\label{eq:2.2}
A_n(x,  y\mid \alpha,  \beta)=\sum_{\sigma\in S_{n+1}}x^{\des(\sigma)-1}y^{\asc(\sigma)-1}\alpha^{\lrmin(\sigma)-1}\beta^{\rlmin(\sigma)-1}. 
\end{align}

A grammar to generate $A_n(x,  y\mid \alpha,  \beta)$ 
was found by Ji \cite{Ji-2025}.

\begin{theorem}[Ji \cite{Ji-2025}]
\label{thm:Eulerexten}
Define a grammar $E$ by
\begin{align}
\label{eq:2.3}
E = \{a \to \alpha a x,  \; b \to \beta b y,  \; x \to x y,  \; y \to x y\}, 
\end{align}
and let  $D_E$ be the formal derivative of $E$. 
For $n\geq 0$, we have  
\begin{align}  
    D_E^n(ab) = a    b \,    A_n(x,  y \mid \alpha,  \beta). 
\end{align} 
\end{theorem}

A grammatical calculus for the
generating function of $A_n(x,  y\mid \alpha,  \beta)$
was given by Ji \cite{Ji-2025}. 

The following proof presents an interplay between two formal derivatives.

\noindent 
\textit{Grammatical Proof of Theorem  \ref{thm:carbi}.} 
\quad 
Let 
\[ \gen_E(ab,t) = \sum_{n=0}^\infty 
D_E^n(ab) \frac{t^n}{n!}. \]
Theorem \ref{thm:Eulerexten} asserts that
\begin{align}\label{eq:3.15}
\gen_E(ab, t) = ab \sum_{n\ge 0} A_n(x,  y \mid \alpha,  \beta)\frac{t^n}{n!}.
\end{align}
Let \(D\) be the formal derivative of the
grammar (\ref{ge}),  where $\alpha$ and $\beta$ 
are regarded as constants. Thus,
\begin{align}\label{eq:3.16}
D(y^\alpha) = \alpha y^{\alpha-1}D(y) = \alpha y^\alpha x,  \qquad
D(x^\beta) = \beta x^{\beta-1}D(x) = \beta x^\beta y.
\end{align}

Let \(a = y^\alpha\) and
\(b = x^\beta\). By  (\ref{eq:3.16}), we find that
\[
D(y^\alpha) = D_E(a),  \quad
D(x^\beta) = D_E(b).
\]
Evidently,
\[
D(x)=D_E(x) ,  \quad
D(y)=D_E(y) .
\]

By induction, we see that 
 for \(n\ge 0\),  
\[
D^n(y^\alpha) = D_{{E}}^n(a),  \quad
D^n(x^\beta) = D_{{E}}^n(b),  
\]
where \(a = y^\alpha\) and  \(b = x^\beta\). 
It follows that for \(n\ge 0\),  
\begin{align}\label{eq:3.17}
D_{E}^n(ab) = D^n(x^\beta y^\alpha),
\end{align}
which yields
\begin{align}\label{eq:3.18}
\gen_E(ab, t)=\gen(x^\beta y^\alpha , t).
\end{align}
By the multiplicative property, the following relation
\begin{equation}
    \label{ab}
\gen(x^\beta y^\alpha , t) = \bigl(\gen(x, t)\bigr)^\beta \bigl(\gen(y, t)\bigr)^\alpha \end{equation} 
is valid when $\alpha$ and $\beta$ are integers.
Moreover, expressing a power series $f(x)^\alpha$ as
$e^{\alpha\log f(x)}$, we see that (\ref{ab}) holds 
when $\alpha$ and $\beta$ are regarded as 
parameters.

Along the lines of the proof of
Theorem \ref{thm-a-n}, we deduce that    
\begin{align}\label{eq:3.7}
\gen(y, t)  = \sum_{n\ge 0} D^n(y)\frac{t^n}{n!} = y\bigl(1+xF(x,  y;t)\bigr),  
\end{align}
and 
\begin{align}\label{eq:3.8}
\gen(x, t)  = \sum_{n\ge 0} D^n(x)\frac{t^n}{n!} = x\bigl(1+yF(x,  y;t)\bigr),
\end{align}
where \(F(x,  y;t)\) is given by (\ref{eq:1.8}).  
Substituting
(\ref{eq:3.7}) and (\ref{eq:3.8}) into (\ref{eq:3.18}) gives
\[
\gen_E (ab, t)=y^\alpha x^\beta \bigl(1+xF(x,  y;t)\bigr)^\alpha \bigl(1+yF(x,  y;t)\bigr)^\beta.  
\]
Comparing with \eqref{eq:3.15} yields \eqref{eq:1.7}. This completes the proof. 
\qed

Next we present a grammatical calculus for the
enumeration of permutations with respect to the three kinds of peaks.
Let $n \geq 1$ and 
let $\sigma = \sigma_1 \sigma_2 \cdots \sigma_n$ 
be a permutation of $[n]$. Assume that
$\sigma_0 = \sigma_{n+1} = 0$. 

An index \(i\) is called
\begin{itemize}
\item a left peak if \(1\le i\le n-1\) and
      \(\sigma_{i-1}<\sigma_i>\sigma_{i+1}\);
\item an interior peak if \(2\le i\le n-1\) and
      \(\sigma_{i-1}<\sigma_i>\sigma_{i+1}\);
\item an exterior peak, or outer peak, if \(1\le i\le n\) and \(\sigma_{i-1}<\sigma_i>\sigma_{i+1}\);
\item a valley if $2 \leq i \leq n-1$ and $\sigma_{i-1} >\sigma_i < \sigma_{i+1}$ . 
\end{itemize}

Let 
  \(L(\sigma)\), \(M(\sigma)\), \(W(\sigma)\) and
  \(V(\sigma)\) 
  denote the numbers of left peaks,
  interior peaks, exterior peaks and valleys of
  \(\sigma\), respectively. Such
  pictographic symbols \(L(\sigma)\), \(M(\sigma)\) and \(W(\sigma)\) were proposed in  \cite{Chen-Fu-2024}. 
  Here   the letter \(L\) resembles a peak on the left,
  the letter $M$ resembles peaks in the middle, and
  the letter $W$ indicates that we allow peaks at the
  beginning and at the end. For example,
  let 
  $\sigma = 7\,  1\,  3\,  8\,  5\,  9\,  6\,  2\,  4$, 
  then we have
\[
L(\sigma) = 3,\;  M(\sigma) = 2, \; W(\sigma) = 4, \; V(\sigma)=3.
\]
It is clear that
\begin{align*}
    W(\sigma) -1 = V (\sigma) = M (\sigma^c),
\end{align*}
where $\sigma^c$ is the complement of $\sigma$, as defined by \eqref{def:com}.

For $n \geq 1$ and  $0 \leq k \leq \lfloor n/2 \rfloor$, 
let $L(n,   k)$ be the
number of permutations of $[n]$ with  $k$ 
left peaks. Similarly,
we may define $M(n,   k)$ and $W(n,   k)$, 
with appropriate ranges of $k$. 
For $n = 0$,   set  $L(0,   0) = M(0,   0) = W(0,   0) = 1$. 

Define the polynomials
\begin{align*}
 L_n(x)& = \sum_{k=0}^{\lfloor n/2 \rfloor} L(n,  k)\,x^k,\\[6pt]
     M_n(x) &  = \sum_{k=0}^{\lrf{(n-1)/2}} M(n, k)\,  x^k,
     \\[6pt]
    W_n(x) & = \sum_{k=1}^{\lrf{(n+1)/2}} W(n, k)\,  x^k.
\end{align*}
Let $L(x,  t)$,  $M(x,  t)$ and  $W (x,  t)$ denote
the exponential generating functions of 
$L_n (x)$,  $ M_n (x)$ and $W_n (x)$, respectively.
The generating function of the interior
peak polynomials $ M_n (x)$ was derived by 
David-Barton \cite{David-Barton} by establishing
a partial differential equation. Equivalent 
formulas can be found in Entringer \cite{Entringer}, Kitaev \cite{Kitaev}  and Zhuang \cite{Zhuang-2016}. It is
easy to deduce $W(x,  t)$ 
from $M(x,  t)$. Gessel~\cite{OEIS} found an explicit
formula for $L(x,  t)$. 

\begin{theorem}[Gessel \cite{OEIS}]
   \label{lem1}
  We have
  \begin{align}\label{eq:gessel}
  \sum_{n \ge 0}L_n(x)\,\frac{t^n}{n!}=\frac{\sqrt{1-x}}{\sqrt{1-x} \cosh(t\sqrt{1-x}) - \sinh(t\sqrt{1-x})}.
\end{align}
\end{theorem}

To set up the grammatical calculus for 
peak polynomials, we need the bivariate forms. 
For $n = 0$,    define $L_0(x,   y) = x$,   $M_0(x,   y) = 1$ 
and $W_0(x,   y) = y$. For $n \geq 1$, define 
\begin{align}
L_n(x,   y) &= \sum_{k=0}^{\lfloor n/2 \rfloor} L(n,   k) \, x^{2k+1} y^{n-2k},   \\[6pt]
M_n(x,   y) &= \sum_{k=0}^{\lfloor (n-1)/2 \rfloor} M(n,   k) \,x^{2k+2} y^{n-2k-1},   \\[6pt]
W_n(x,   y) &= \sum_{k=1}^{\lfloor (n+1)/2 \rfloor} W(n,   k)\, x^{2k} y^{n-2k+1}.
\end{align}

Consider the grammar
\begin{align}\label{Grammar_left}
G = \{x \to xy,   \ y \to x^2\}. 
\end{align}
Let $D$ be the formal derivative of the
grammar \eqref{Grammar_left}. We have the 
following grammatical expressions for
the three kinds of peak polynomials. 
For $n \geq 0$,   we have 
\begin{align*}
D^n(x) = L_n(x,   y),   \quad D^n(y) = W_n(x,   y).
\end{align*}
For $n \geq 1$, we have $D^n(y) = M_n(x,   y)$. 
The grammar (\ref{Grammar_left}) 
together with 
the relation $D^n(x)=L_n(x,  y)$
was independently obtained by Ma \cite{Ma-2012} 
and Chen-Fu~\cite{Chen-Taiwan, Chen-Fu-2017}.
A grammatical labeling was given in
\cite{Chen-Fu-2017}. 

The generating function of $L_n(x,y)$ 
can be derived using the grammatical calculus, see Chen-Fu \cite{Chen-Fu-2023-JACO}.

\begin{theorem}[Chen-Fu \cite{Chen-Fu-2023-JACO}]
\label{thm4.5}
 Let $L_0(x,  y)=x$. We have
\begin{align} 
\sum_{n=0}^{\infty}L_n(x,  y)\frac{t^n}{n!}= \frac{x\sqrt{y^2 - x^2}}{\sqrt{y^2 - x^2} \cosh(t\sqrt{y^2 - x^2}) - y \sinh(t\sqrt{y^2 - x^2})}. \label{eq:2.22}
\end{align}
\end{theorem}

\begin{proof}
Let $D$ be the formal derivative of the grammar \eqref{Grammar_left}. 
For $n \ge 0$,  we have $D^n(x) = L_n(x,   y)$. 
We proceed to compute $\gen(x,  t)$. 
Notice that
\[
D(x^{-1}) = -x^{-1}y,   \quad D^2(x^{-1}) = x^{-1}(y^2 - x^2). 
\]
It follows that for $n \ge 0$,   
\begin{align}
D^{2n}(x^{-1}) &= x^{-1}(y^2 - x^2)^n,   \label{eq:2.18} \\[3pt]
D^{2n+1}(x^{-1}) &= -x^{-1}y(y^2 - x^2)^n. \label{eq:2.19}
\end{align}
Now we have
\begin{align}
\sum_{n=0}^{\infty} D^{2n}(x^{-1}) \frac{t^{2n}}{(2n)!} &= x^{-1} \cosh(t\sqrt{y^2 - x^2}),   \label{eq:2.20} \\
\sum_{n=0}^{\infty} D^{2n+1}(x^{-1}) \frac{t^{2n+1}}{(2n+1)!} &= -\frac{x^{-1}y}{\sqrt{y^2 - x^2}} \sinh(t\sqrt{y^2 - x^2}). \label{eq:2.21}
\end{align}
Combining (\ref{eq:2.20}) and (\ref{eq:2.21}), we deduce that
\begin{align}
\gen(x^{-1},  t) = \frac{\sqrt{y^2 - x^2} \cosh(t\sqrt{y^2 - x^2}) - y \sinh(t\sqrt{y^2 - x^2})}{x\sqrt{y^2 - x^2}}, \label{eq:2.17}
\end{align}
or equivalently,
\begin{align*}
    \gen(x,  t) =\frac{1}{\gen(x^{-1},  t)}=\frac{x\sqrt{y^2 - x^2}}{\sqrt{y^2 - x^2} \cosh(t\sqrt{y^2 - x^2}) - y \sinh(t\sqrt{y^2 - x^2})},
\end{align*}
as required. \qed
\end{proof}

Dividing both sides of  \eqref{eq:2.22} by $x$,  replacing
$x$ by $x^2$ and setting $y = 1$,   we come to Gessel's formula \eqref{eq:gessel}.     

Given a permutation
$\sigma = \sigma_1 \cdots \sigma_n \in S_n$, 
assume that $\sigma_0 = 0$ and $\sigma_{n+1} = +\infty$. 
We now define the following terms. 
\begin{itemize}
    \item For $1  \leq i < n$, if $\sigma_{i-1} < \sigma_i < \sigma_{i+1}$,  then  $i$ is called a left double ascent of 
    $\sigma$.  
    \item For  $1 < i < n$, if  $\sigma_{i-1} < \sigma_i < \sigma_{i+1}$, then  $i$ is called a double ascent of
    $\sigma$. 
    \item For $1 < i \leq n$,
    if $\sigma_{i-1} < \sigma_i < \sigma_{i+1}$, then $i$ is
    called a right double ascent of $\sigma$.

    \item For  $1 \leq i \leq n$,
    if $\sigma_{i-1} < \sigma_i < \sigma_{i+1}$, then  $i$
    is called a left-right double ascent of $\sigma$.  
\end{itemize}

Given a permutation
$\sigma = \sigma_1 \cdots \sigma_n \in S_n$,
the following patterns are defined under
the convention that $\sigma_0 = +\infty$ and $\sigma_{n+1} = 0$, 

\begin{itemize}
    \item For $1 \leq i < n$, if $\sigma_{i-1} > \sigma_i > \sigma_{i+1}$, then $i$ is called a left
    double descent of $\sigma$. 
    
    \item For  $1 < i < n$, if $\sigma_{i-1} > \sigma_i > \sigma_{i+1}$, then $i$ is called a double
    descent of $\sigma$.
    
    \item For $1 < i \leq n$,
    if $\sigma_{i-1} > \sigma_i > \sigma_{i+1}$, then
    $i$ is called a right double descent of $\sigma$.
    
    \item For $1 \leq i \leq n$,
    if $\sigma_{i-1} > \sigma_i > \sigma_{i+1}$,  then
    $i$ is called a left-right double descent of
    $\sigma$. 
\end{itemize}

The number of indices in the above definitions will be 
denoted by concatenation of the initials of the terms. For example, ${\rm lrdd}(\sigma)$ stands for the number of left-right double descents
of $\sigma$. 
Denote by $U(n,  k)$ the number
of permutations of $[n]$ with $k$ double descents.
Define 
 \begin{align}\label{yny}
 U_n(y)=\sum_{k\geq 0}U(n,  k)\, y^k,
 \end{align}
 that is, $U_n(y)$ is the generating function of 
 ${\rm dd}(\sigma)$ over $S_n$. 
 Elizalde-Noy \cite{Elizalde-2003} 
 derived the generating function of $U_n(y)$ 
 via an ordinary differential equation.
 
 \begin{theorem}[Elizalde-Noy \cite{Elizalde-2003}]\label{lem2}
We have
\begin{align}\label{gf3}
\sum_{n=0}^\infty U_n(y)\frac{t^n}{n!}
=\frac{2\sqrt{(y-1)(y+3)}e^{t/2\cdot(1-y+\sqrt{(y-1)(y+3)})}}{1+y+\sqrt{(y-1)(y+3)}-(1+y-\sqrt{(y-1)(y+3)})e^{t\sqrt{(y-1)(y+3)}}}.
\end{align}
\end{theorem}

For the special case $y=0$,
Barry \cite{Barry-2014} and Basset \cite{Basset-2014} 
independently obtained the following formula.

\begin{theorem}[Elizalde-Noy \cite{Elizalde-2003},   Barry \cite{Barry-2014},   Basset \cite{Basset-2014}]\label{lem3}
We have
  \begin{align}\label{No-doubledescents-formula}
  \sum_{n=0}^{\infty}U(n,  0)\frac{t^n}{n!}
    =\frac{\sqrt{3}}{2}\frac{e^{t/2}}{\cos{(\sqrt{3}t/2+\pi/6)}}.
  \end{align}
\end{theorem}

Fu \cite{Fu-2018} presented a grammatical 
calculus for the joint distribution of the
number of left peaks and the number of double 
descents over permutations of $[n]$. Consider the
grammar
\begin{align}
G=\{ x \rightarrow xy,  \; y \rightarrow xz,  \; z \rightarrow zw,  \; w \rightarrow xz\}.
\end{align}
It can be seen that
this grammar is a unification 
of the grammar (\ref{ge}) for the Eulerian polynomials
and the grammar (\ref{Grammar_left}) for the number 
of left peaks. Substituting $w$ with
$x$ and $z$ with $y$,   
$G$ becomes (\ref{ge}). It should be noted that
such substitutions are consistent with the original grammar.
On the other hand, substituting $z$ with $x$ and 
$w$ with $y$,   
 $G$ becomes the grammar (\ref{Grammar_left}).
Let  $V_n(i,   j)$ denote 
the number of permutations of $[n]$ with
$i$ left peaks and  $j$ double descents. 
Let $V_0(x,  y,  z,  w)=z$.  For $n\ge1$,   define
\begin{align}\label{eq4}
V_n(x,  y,  z,  w)=\sum_{i,  j}
V_n(i,   j)\,x^i y^j z^{i+1} w^{n-2i-j},  
\end{align}
where the sum ranges over $i$ and $j$ 
subject to $2i+j\leq n$.

Fu \cite{Fu-2018} gave a grammatical 
labeling of permutations and obtained
the generating function of $V_n(x,  y,  z,  w)$ by
the grammatical calculus. 

\begin{theorem}[Fu \cite{Fu-2018} ]\label{thm2}
  We have
  \begin{align*}
  &\sum_{n\ge0}V_n(x,   y,  z,   w)\frac{t^n}{n!}\\
  &\qquad =\frac{2z\sqrt{(w + y)^2 - 4xz}
  e^{t/2 \cdot (w-y+\sqrt{(w + y)^2 - 4xz})}}
  {w+y+\sqrt{(w + y)^2 - 4xz}-
  (w+y-\sqrt{(w + y)^2 - 4xz})e^{t\sqrt{(w + y)^2 - 4xz}}}.
  \end{align*}
\end{theorem}

Theorem \ref{thm2} has several 
specializations. Setting $y=z=w=1$,  
   it reduces to Theorem \ref{lem1}. 
   Setting $z=x$ and $w=y$, we are led to 
   Theorem \ref{thm4.5}. Setting $x=z=w=1$
   yields Theorem \ref{lem2}. Setting
   $x=z=w=1$ and $y=0$, we come to Theorem \ref{lem3}.

Carlitz-Scoville \cite{Carlitz-Scoville-1974} 
investigated the joint distribution of several
statistics of permutations, and obtained 
the following generating function, as restated by
Stanley \cite{Stanley-I}.

\begin{align}
\label{css}
\sum_{n \ge 1} \left( \sum_{\sigma \in S_n} u^{V(\sigma)} v^{W(\sigma)-1} w^{\lda(\sigma)} z^{\rdd(\sigma)} \right) \frac{t^n}{n!} = F(x, y; t), 
\end{align}
where $x + y = w + z$, $xy = uv$,  and \(F(x,  y;t)\) is
given by (\ref{eq:1.8}).
Fu \cite{Fu-2018} presented a grammatical
labeling along with the grammatical
calculus for the above formula. 
The original formulation of 
Carlitz-Scoville can be 
recovered by differentiating (\ref{css})
with respect to $t$.
For $n \ge 1$, define 
    \begin{align}\label{pnb}
    P_n(u, v, w, z) = \sum_{\sigma \in S_{n+1}} u^{V(\sigma)} v^{W(\sigma)-1} w^{\rdd(\sigma)} z^{\lda(\sigma)}. 
    \end{align}
For $n=0$, set $P_0(u, v, w, z) = 1$.
Notice that the above definition (\ref{pnb})
aligns with the definition (\ref{eq4}).

\begin{theorem}[Carlitz-Scoville \cite{Carlitz-Scoville-1974}]
\label{thm:4.9}
    We have
    \begin{align}
    \sum_{n \ge 0} P_n(u, v, w, z) \frac{t^n}{n!} = (1 + yF(x, y; t))(1 + xF(x, y; t)), 
    \end{align}
    where $x + y = w + z$, $xy = uv$,  and \(F(x,  y;t)\) is
    given by (\ref{eq:1.8}).
\end{theorem}

Ji \cite{Ji-2025} introduced the \((\alpha,  \beta)\)-extension of the above polynomials $P_n(u,  v,  w,  z)$, \
that is, for $n\geq 1$,   define
\begin{align}\label{defi-P}
P_n(u,  v,  w,  z\mid \alpha,  \beta)
=\sum_{\sigma\in S_{n+1}}
u^{\mathrm{V}(\sigma)}v^{W(\sigma)-1}w^{\rdd(\sigma)}z^{\lda(\sigma)}
\alpha^{\lrmax(\sigma)-1}\beta^{\rlmax(\sigma)-1}.
\end{align} 
By convention, set \(P_0(u,  v,  w,  z\mid \alpha,  \beta)=1\).
The first few values of \(P_n(u,  v,  w,  z\mid \alpha,  \beta)\) are as follows:
\begin{align*}
P_1(u,  v,  w,  z\mid \alpha,  \beta)
&=\alpha z+\beta w;\\[5pt]
P_2(u,  v,  w,  z\mid \alpha,  \beta)
&=\alpha^2 z^2+2\alpha\beta wz+\beta^2 w^2+\alpha vu+\beta vu;\\[5pt]
P_3(u,  v,  w,  z\mid \alpha,  \beta)
&=\alpha^3 z^3+3\alpha^2\beta wz^2+3\alpha\beta^2 w^2z
+3\alpha\beta uvz+3\alpha^2 uvz+\alpha uvz\\[5pt]
&\quad +\beta uvz+\beta^3 w^3+3\beta^2 uvw+\alpha uvw+\beta uvw+3\alpha\beta uvw.
\end{align*}

By taking complementation of permutations, we have
the following equivalent definition:
\begin{align}\label{eq:4.49}
P_n(u,  v,  w,  z \mid \alpha,  \beta) = \sum_{\sigma\in S_{ n+1}} (uv)^{M(\sigma)}w^{\rda(\sigma)}z^{\ldd(\sigma)}\alpha^{\lrmin(\sigma)-1}\beta^{\rlmin(\sigma)-1}.
\end{align}

The following grammar for \(P_n(u,  v,  w,  z \mid \alpha,  \beta)\) was given by Ji \cite{Ji-2025}: 
\[
{H} = \{a \to \alpha a z,  \; b \to \beta b w,  \; z \to uv,  \; w \to uv,  \; u \to uw,  \; v \to vz\} 
\]

\begin{theorem}[Ji \cite{Ji-2025}]\label{thm:fourvaria}
Let \(D_{{H}}\) be the formal derivative of
the above grammar $H$. For \(n \ge 0\),  we have
\[
D_{{H}}^n(ab) = ab P_n(u,  v,  w,  z \mid \alpha,  \beta). 
\]
\end{theorem}

Employing Theorem \ref{thm:fourvaria},    Ji \cite{Ji-2025} 
obtained the following generating function.

\begin{theorem}[Ji \cite{Ji-2025}]
\label{thm:Carlitz-ji}
We have
\[
\sum_{n\geq 0}P_n(u,  v,  w,  z\mid \alpha,  \beta)\frac{t^n}{n!}
=\bigl(1+yF(x,  y;t)\bigr)^{\frac{\alpha+\beta}{2}}
\bigl(1+xF(x,  y;t)\bigr)^{\frac{\alpha+\beta}{2}}
e^{\frac{1}{2}(\beta-\alpha)(w-z)t},  
\]
where \(x+y=w+z\),  \(xy=uv\),  and \(F(x,  y;t)\) is 
given by  (\ref{eq:1.8}). 
\end{theorem}
  
There are several consequences of Theorem \ref{thm:Carlitz-ji}. 
First, we get an explicit expression for
the generating function of 
\(P_n(u,  v,  w,  z\mid \alpha,  \beta)\). 

\begin{theorem}[Ji \cite{Ji-2025}] \label{thm:Carlitz-ji-main2} We have
\begin{align*}
&\sum_{n\geq 0} P_n(u ,  v,  w,  z ~|~{\alpha},  {\beta}) \frac{t^n}{n!}=e^{\frac{1}{2}({\beta}-{\alpha})(w-z )t}\times \left(\cosh\left(\frac{t}{2}\sqrt{(w+z )^2  -4u v}\right) \right.\\[3pt]
&   \qquad \qquad \qquad \qquad \qquad\qquad\qquad \left. -\frac{w+z }{\sqrt{(w+z )^2-4u v}}\sinh\left(\frac{t}{2}\sqrt{(w+z )^2-4u v}\right)\right)^{-({\alpha}+{\beta})}.
\end{align*}
\end{theorem}

 When $\alpha=\beta=1$, Theorem \ref{thm:Carlitz-ji} 
 reduces to Theorem \ref{thm:4.9}. 
Setting \(\alpha = 0\) in (\ref{eq:4.49}),  we get 
\begin{align}\label{eq:2.24}
P_n(u,  v,  w,  z \mid 0,  \beta) = \sum_{\sigma\in S_n} (uv)^{L(\sigma)}w^{\lrda(\sigma)}z^{\dd( \sigma)}\beta^{\rlmin(\sigma)}.
\end{align}
 In view of  (\ref{eq:2.24}), a
 $\beta$-extension of Theorem \ref{thm2} 
 can be deduced from Theorem \ref{thm:Carlitz-ji-main2} by setting $\alpha=0$.

Gessel-Stanley \cite{Gessel-Stanley-1978}
introduced the notion of Stirling
permutations; see Section \ref{sec:9}. The Eulerian polynomials
(with respect to the number of descents)
over Stirling permutations are referred to as
the second-order Eulerian polynomials, which can
be generated by the following grammar 
\[   x \rightarrow xy^2,  \quad y \rightarrow xy^2 , \]
see \cite{Chen-Fu-2017}. The higher order 
Eulerian polynomials over $r$-Stirling
permutations can be generated by the grammar
\[ x \rightarrow xy^r,  \quad y \rightarrow 
xy^r. \]

\section{Grammar Assisted Bijections}

   In this section, 
   we use two examples to illustrate how
    grammars can be instrumental in constructing 
   bijections. The first example is a bijection
   between permutations and increasing 
   trees that maps descents of a permutation to 
   the leaves of an increasing tree.
   
   Dumont \cite{Dumont-1996} found the
   following grammar for increasing trees: 
\begin{align}
x_i \rightarrow x_0\,x_{i+1},   \quad i\geq 0,  
\label{Dumont-tree}
\end{align}
where $x_i$ is the label of a vertex
with outdegree $i$ (the number of its children). 
Let $T$ be an increasing tree on 
$\{0,  1,  2,  \dots,  n\}$. When the new vertex $n+1$ 
is attached to $T$, some labels will be updated.
Assume that $n+1$ is inserted to $T$ as a child of a vertex
$v$. If $v$ has outdegree $i$, then the label of $v$ will
be updated by $x_{i+1}$ and the leaf $n+1$ will be labeled
by $x_0$. This change of 
labels is captured by the rule  $x_i \to x_0 \,x_{i+1}$.
Let $D$ be the 
formal derivative of the grammar \eqref{Dumont-tree}. Dumont \cite{Dumont-1996} showed that 
\begin{equation}\label{tw}
D^n(x_0) = \sum_{T} x_{\deg_T(0)}\,
x_{\deg_T(1)} \, \cdots \,x_{\deg_T(n)},
\end{equation}
where the sum ranges over increasing trees $T$ on $\{0,  1,  2,  \dots,  n\}$  and $\deg_T(i)$ denotes 
the outdegree of $i$ in $T$.  

If we use different labels to distinguish
internal vertices and leaves of an increasing tree,
and, more precisely, use $x$ to label
an internal vertex and use $y$ to label
a leaf, then we are led to the grammar 
$$x \rightarrow xy,   \quad y \rightarrow xy,$$
which is exactly the grammar for the 
Eulerian polynomials. This means that
there is a combinatorial interpretation of 
the Eulerian polynomials 
in terms of increasing trees, see Dumont \cite{Dumont-1996}.
A generating function proof of this 
fact was given by 
Lin-Liu-Wang-Zang \cite{Lin-Liu-Wang-Zang-2024}.

\begin{theorem}[Dumont \cite{Dumont-1996}]
    For  $n \geq 1$,  
    the number of permutations of $[n]$ 
    with $k$ descents equals the 
    number of increasing trees on $\{0,   1,   2,   \ldots,   n\}$ with  $k$ leaves. 
\end{theorem}

Guided by the grammar, it is easy to 
construct a bijection for the above correspondence.

Following this line of reasoning,
Chen-Fu \cite{Chen-Fu-2017} obtained the
following correspondence.  

\begin{theorem}[Chen-Fu \cite{Chen-Fu-2017}]
\label{thm:lpincrea} 
 The number of permutations
 of $[n]$ with $k$ left peaks equals 
 the number of increasing trees on $\{0,  1,  2,  \dots,  n\}$
 with $2k+1$ vertices of even degree. 
\end{theorem}

The grammatical justification takes only one line. 
In the grammar  (\ref{Dumont-tree}),   
setting $x_{2i} = x$ and $x_{2i+1} = y$,  we 
are led to the 
grammar (\ref{Grammar_left}) for the
number of left peaks of a permutation.
This gives a confirmation of Theorem 
\ref{thm:lpincrea}. A grammar assisted
bijection is provided by Chen-Fu \cite{Chen-Fu-2023-JACO}. 
 
As for exterior peaks,  Chen-Fu \cite{Chen-Fu-2023-JACO} 
obtained the following correspondence and found a 
grammar assisted bijection.

\begin{theorem}[Chen-Fu \cite{Chen-Fu-2023-JACO}]
For \(n\ge1\), there is a 
bijection from the set of permutations of 
\([n]\) to the set of increasing trees 
on \(\{0,1,\ldots,n\}\) that maps the
number $k$ of exterior peaks of a permutation
to the number $j$ of nonroot vertices of even degree
such that  $k = \lrf{(j + 1) / 2}$.
\end{theorem}

There are two kinds of runs of a permutation.
First, after patching a zero to the
beginning of a permutation $\sigma=\sigma_1\sigma_2 \cdots \sigma_n$, that is, setting $\sigma_0 = 0$,
define an up-down run to be 
a maximal ascending segment or a maximal
descending segment (of consecutive elements). 
If we drop the assumption that $\sigma_0 = 0$, 
a maximal ascending or descending segment
is called an alternating run. For example,
the permutation $21$ has two up-down runs $0\,  2$ and $2\,  1$),  but it has only one alternating run $2\,  1$. 

Using a recurrence relation, 
Ma  \cite{Ma-2013} discovered a grammar for up-down runs,
that is, 
\[G=\{a \to ax,  \; x \to xy,  \; y \to x^2\}.\]
A grammatical labeling for up-down runs 
was found by 
Chen-Fu \cite{Chen-Fu-2023-JACO}.
This labeling scheme demonstrates 
how an insertion of the element
$n+1$ to a permutation of $[n]$ affects 
the number of up-down runs. This labeling
gives a clue to establish
a correspondence between 
permutations and increasing trees 
that maps the number of up-down runs to
the number of vertices in an increasing
tree subject to a certain degree constraint.

\begin{theorem}[Chen-Fu \cite{Chen-Fu-2023-JACO}]
   For  $n \ge 1$,   there is a bijection
   from the set of permutations of $[n]$ 
   to the set of increasing trees on 
   $\{0,   1,   \dots,   n\}$ 
   that maps the number of up-down
   runs of a permutation 
   to the number of nonroot vertices
   of even degree.  
\end{theorem}

The second example exhibits a connection
between a grammar of Dumont and a theorem
of Diaconis-Evans-Graham on permutation statistics. 
Let $\sigma = \sigma_1 \cdots \sigma_n \in S_n$. 
An index $1 \leq i \leq n$ 
is called an excedance if $\sigma_i > i$,
a drop if $\sigma_i < i$, and a fixed 
point if $\sigma_i = i$. Obviously,
$n$ cannot be an excedance and $1$ cannot be a drop. 
The numbers of excedances, drops and 
fixed points of $\sigma$ are denoted by 
$\exc(\sigma)$,  $\drop(\sigma)$ and $\fix(\sigma)$, 
respectively. A drop is also called an anti-excedance.

Dumont \cite{Dumont-1996} obtained the
following grammar for the joint distribution 
$(\exc,   \drop,   \fix)$ over permutations:
\begin{align}
\label{gram:dumont}
G = \{a \to az,   z \to xy,   x \to xy,   y \to xy\}.
\end{align}
Define
$$F_n(x,   y,   z) = \sum_{\sigma \in S_n} x^{\exc(\sigma)} y^{\drop(\sigma)} z^{\fix(\sigma)},  $$ 
and let $F_0(x,   y,   z) = 1$.

\begin{theorem}[Dumont \cite{Dumont-1996}]
  Let $D$ be the formal derivative of 
  the grammar (\ref{gram:dumont}). For $n \ge 0$,   we have
  \[ D^n(a) = a F_n(x,   y,   z).\] 
\end{theorem} 

It should be noted that the
work of Roselle   \cite{Roselle-1968} was
not accurately cited   by 
Dumont \cite{Dumont-1996}. However,
the pointer to the work of Roselle turned out
to be unexpectedly informative. Namely, Roselle \cite{Roselle-1968} in fact investigated the
 joint distribution of the number of
 ascents and the number of successions.
 
 Assume that $n\geq 1$ and $\sigma= 
\sigma_1 \sigma_2\cdots \sigma_n$ is
a permutation of $[n]$.
An index $1\leq i \leq n-1$ is called a succession if $\sigma_{i+1}=\sigma_{i}+1$ and the number of
successions of $\sigma$ is denoted by $\suc(\sigma)$. 
Roselle found the generating function for
the joint distribution of $(\asc,  \suc)$ over
permutations. A related statistic is called the 
big ascent by Ma-Qi-Yeh-Yeh \cite{MQYY-2025}.
If $1\leq i \leq n-1$ and $\sigma_{i} +2 
\leq \sigma_{i+1}$,  then $i$ is called a big ascent of
$\sigma$. The number of big ascents of $\sigma$
is denoted by $\basc(\sigma)$. The generating function of the
joint distribution of $(\basc,  \des,  \suc)$ over
permutations has been derived in 
\cite{MQYY-2025}. 

While the grammar of Dumont does not
seem to apply to the joint 
distribution of $(\asc, \suc)$,  Chen-Fu \cite{Chen-Fu-2024} 
observed that  it instead applies 
to the joint distribution of $(\exc, \fix)$, as 
well as to the joint distribution of $(\jump, \lsuc)$, where
the statistics  $\jump$ and $\lsuc$ are defined as follows.
For a permutation $\sigma
=\sigma_1\sigma_2\cdots\sigma_n$,  assume that $\sigma_0=0$.
 An index $ 1\leq i \leq n $ is called a left 
 succession if $\sigma_{i-1}+1=
\sigma_{i}$, and a jump if 
$\sigma_{i-1}+2 \leq \sigma_i$. The numbers of
left successions and jumps of $\sigma$ are denoted by
$\lsuc(\sigma)$ and $\jump(\sigma)$, respectively. 

For  $n \ge 1$ and a permutation $\sigma \in S_n$,  define
\begin{align}
    S(\sigma) &= \{i \mid 1 \leq i \leq n-1,   \sigma_i + 1 = \sigma_{i+1}\},  \\
    G(\sigma) &= \{i \mid 1 \leq i \leq n-1,   \sigma_i = i\},  \\
    F(\sigma) &= \{i \mid 1 \leq i \leq n,   \sigma_i = i\}.
\end{align}
It should be noted that the index $n$ is excluded 
in the definition of $G(\sigma)$.
Given a subset $I \subseteq [n-1]$,  let $M_n(I)$ 
denote the set of permutations $\sigma \in S_n$
such that $S(\sigma) = I$, and  let 
$G_n(I)$ denote the set of permutations 
$\sigma \in S_n$ such that $G(\sigma) = I$,
and let 
$F_n(I)$ denote the set of permutations
$\sigma \in S_n$ such that $F(\sigma) = I$.
  Diaconis-Evans-Graham \cite{Diaconis-Evans-Graham} 
  discovered the following correspondence.

\begin{theorem}[Diaconis-Evans-Graham \cite{Diaconis-Evans-Graham}]
\label{DEG}
Let $n \geq 1$ and $I \subseteq [n-1]$. 
Then there exists a one-to-one correspondence 
between $M_n(I)$ and $G_n(I)$. 
\end{theorem} 

For $n \geq 1$ and $\sigma \in S_n$, define
$$\overline{S}(\sigma) = \{\sigma_i \mid 1 \leq i \leq n - 1,   \sigma_i + 1 = \sigma_{i+1}\}.$$
It is easily seen that for any $\sigma \in S_n$,
\begin{align}
    \overline{S}(\sigma^{-1}) &= S(\sigma),   \\
    G(\sigma^{-1}) &= G(\sigma),  \\
    F(\sigma^{-1}) &= F(\sigma),  
\end{align}
where $\sigma^{-1}$ stands for the inverse permutation
$\sigma$. 
For  $n \geq 1$ and
a subset $I \subseteq [n-1]$,  let
$\overline{S}_n(I)$ denote the set of permutations 
$\sigma \in S_n$ such that
$\overline{S}(\sigma) = I$. Then Theorem \ref{DEG} can be
restated as follows, see Chen-Fu \cite{Chen-Fu-2024}.

\begin{theorem} 
   Let $n \geq 1$ and $I \subseteq [n-1]$. 
   There exists a bijection
   between  $\overline{S}_n(I)$ and $G_n(I)$. 
\end{theorem}

Chen-Fu \cite{Chen-Fu-2024} observed the
connection between the grammar \eqref{gram:dumont} of Dumont \cite{Dumont-1996} and Theorem \ref{DEG} of Diaconis-Evans-Graham, and obtained the
left succession analogue along with 
a grammatical labeling and a grammar assisted bijection.

Analogous to  $\overline{S}(\sigma)$,   for  $n \geq 1$ and a 
permutation $\sigma$ of $[n]$,  define
$$\overline{L}(\sigma) = \{\sigma_i \mid 1 \leq i \leq n,   \sigma_{i-1} + 1 = \sigma_i\}.$$
Moreover, for a subset $I$ of $[n]$,  define
$\overline{L}_n(I)$ to be the
set of permutations $\sigma$ of $[n]$
such that $\overline{L}(\sigma) = I$. 

\begin{theorem}[Chen-Fu \cite{Chen-Fu-2024}]
Let $n \geq 1$ and $I \subseteq [n]$. 
There exists a bijection $\phi$
from $\overline{L}_n(I)$ to  $F_n(I)$ such that
$(\jump(\sigma),   \des(\sigma)-1) = (\exc(\phi(\sigma)),   \drop(\phi(\sigma)))$ for any permutation $\sigma$ of $[n]$. 
\end{theorem}

\section{The Expansion  of $(g(x)D)^n f(x)$}

In this section, we 
discuss the expansion of \((g(x)D)^n\) when
applied to a function \(f(x)\), where
\(D=d/dx\), and \(g(x)\) can be viewed as an
operator of multiplication by \(g(x)\). 
Noting that
\[
(Dg(x))(f(x))=(g(x)D)(f(x))+g'(x)(f(x)),
\]
we have the relation
\begin{equation}\label{Dg}
Dg=gD+g',
\end{equation}
where the multiplication operator \(g\) does not commute with  \(D\). By (\ref{Dg}),
we obtain a grammar for the expansion of 
\((g(x)D)^n f\) and we may also 
look at the expansion of \((g(x)D)^n f\) from the
perspective of normal ordering defined below; see \eqref{no}.

The expansion of \((g(x)D)^n f(x)\) has been extensively studied, see Scherk
\cite{Scherk-1823}, Comtet \cite{Comtet-1973},
Mansour-Schork \cite{MS-2016} and Han-Ma \cite{Han-Ma-2024}.
In particular, Carlitz-Scoville \cite{Carlitz-Scoville-1972}
investigated
the expansion of \(((1+x^2)D)^n f(x)\). In fact,
the involved polynomials, called the 
derivative polynomials, can be traced back to
the work of Knuth-Buckholtz \cite{KB-1967} 
on the Euler numbers. 
For the generating functions and the grammatical calculus of
the derivative polynomials, see  
\cite{Chen-Fu-2023-AC, Hoffman}.

Write
\begin{align}\label{Expansion-A}
F_n = (g(x)  D)^n f(x) =   \sum_{k=0}^n A_{n,  k}(g_0,   g_1,   \ldots,   g_{n-k}) f_k,  
\end{align}
where $f_i = D^i (f) $ and $g_i = D^i(g)$.  
The polynomials
$A_{n,  k}(g_0,   g_1,   \ldots,   g_{n-k})$
have been called the 
$A$-polynomials. The coefficients of the $A$-polynomials
are listed as A139605 in OEIS \cite{OEIS}.

Let $f_0=f$,  
 $g_0=g$,   and $A_{0,  0} =1$. 
 For $n\leq 4$,  we have  
\begin{eqnarray*}
F_0 &= &f_0, \\[3pt]
 F_1 &=& g_0f_1, \\[3pt]
 F_2 & = &          g_0g_1f_1 +  g_0^2 f_2 ,  \\[3pt]
F_3 & = &   (g_0^2 g_2 +g_0 g_1 ^2) f_1+ 3  g_0^2 g_1 f_2 +g_0^3f_3,  \\[3pt]
F_4 & = &   
(g_0^3 g_3   + 4 g_0^2 g_1 g_2 +g_0 g_1^3) f_1+(4 g_0^3 g_2  
+7  g_0 ^2g_1^2) f_2 +6 g_0^3 g_1 f_3+g_0^4f_4.
\end{eqnarray*}

The $A$-polynomials first appeared in the
thesis of Scherk in 1823  \cite{Scherk-1823}, 
where  he computed $F_n$ for $n\leq 5$,   
see Blasiak-Flajolet \cite{BF-2011}.
Murphy \cite{Murphy-1837} in 1837 
 obtained explicit expressions for the
coefficients of $f_k$ in $F_n$ for $k=n,n-1, n-2$. 
Comtet \cite{Comtet-1973}
obtained the following formula for general $n$.

\begin{theorem}[Comtet \cite{Comtet-1973}]
\label{Prop:6.1}
For $1\leq k\leq n$,  
we have
\begin{align}\label{Ank-Comtet}
A_{n,  k}=\frac{g_0}{k!}
\sum (2-k_1)(3-k_1-k_2)\cdots (n-k_1-k_2-\cdots-k_{n-1})\frac{g_{k_1}}{k_1!}\cdots\frac{g_{k_{n-1}}}{k_{n-1}!},  
\end{align}
where the sum ranges over
sequences of nonnegative
integers $(k_1,  k_2,  \ldots,  k_{n-1})$ such that
$k_1+k_2+\cdots+k_{n-1}=n-k$ and   $k_1+\cdots+k_j\leq j$ 
for any $1\leq j\leq n-1$. 
\end{theorem}

The idea of using labeled trees to describe 
the action of partial differential operators on multivariate
functions goes back to Cayley.  Following this direction,
Copeland~\cite{Copeland-2010} and 
Blasiak-Flajolet \cite{BF-2011} independently
obtained an interpretation of 
$A_{n, \, k}$ in terms of increasing trees. 

Let $T$ be an increasing tree on 
$\{0, 1, \ldots,n\}$. As before, let 
$\deg_T(i)$ denote the outdegree of $i$ in $T$. Define
the weight of $T$ by 
\[ g_{\deg_T(1)} \, \cdots\,g_{\deg_T(n)}\,f_{\deg_T(0)}. \]

\begin{theorem}[Copeland \cite{Copeland-2010},   Blasiak-Flajolet \cite{BF-2011}] 
\label{thm:Copeland-B-F}
For $n\geq 1$,  
$F_n$ equals the total weight
of increasing trees on $\{0, 1, \ldots,n\}$.  
\end{theorem}

Chen-Fu-Wang \footnote{ \label{fn-2026}
W.Y.C. Chen, A.M. Fu and E.L. Wang,  The grammatical calculus and normal ordering.
Preprint,  2026}
obtained a grammar that generates $F_n$.

\begin{theorem}
Let $D$ be the formal derivative of the following
grammar  
\begin{align} \label{g-Dumont-fg}
   G=\{ f_i \rightarrow g_0\,  f_{i+1},  \;\; g_i \rightarrow g_0\, g_{i+1} \;|\;
   i = 0,   1,   2,   \ldots \}.
\end{align}
Then    $F_n = D^n(f_0)$.
\end{theorem}

Note that the above increasing tree interpretation
of the $A$-polynomials  immediately follows from
the grammar \eqref{g-Dumont-fg}. The increasing tree
interpretation can be reformulated in terms of 
inversion sequences, because  the Pr\"ufer codes
of increasing trees are precisely inversion sequences.

To be more specific, 
a sequence $e = (e_1, e_2, \dots, e_n)$ of   nonnegative integers is called an inversion sequence
if for $1 \le i \le n$,  we have $0 \le e_i < i$.  Han-Ma \cite{Han-Ma-2024} obtained an expansion
of $(gD)^n f$ in terms of inversion sequences and
gave a derivation of Comtet's formula in Theorem \ref{Prop:6.1}.

Next we show how the Pr\"ufer codes or the inversion
codes lead to the formula of Comtet. 
Let $T$ be an increasing tree on 
$\{0,  1,  2,  \ldots,  n\}$.
Notice that the maximum element $n$ must be a leaf.
Let $e_n$ denote the parent of $n$. 
After this leaf is removed,
we get an increasing tree on 
$\{0, 1,2, \ldots, n-1\}$ in which 
$n-1$ must be a leaf.
Denote the parent of $n-1$ by $e_{n-1}$. 
Repeating this procedure yields 
a sequence $(e_n, e_{n-1}, \ldots, e_1)$. 
Reversing this sequence, we get an inversion
code $(e_1,\ldots,e_n)$.
Observe that the number of occurrences of 
$i$ equals the outdegree of the vertex $i$ 
in the increasing tree. Thus, we obtain
\begin{align} \label{x}
\sum_T t_0^{\deg_T(0)}\, t_1^{\deg_T(1)}\, \cdots
\, t_n^{\deg_T(n)} = t_0\,  (t_0+t_1) \,  \cdots \,(t_0+t_1+\cdots+t_{n-1}),
\end{align} 
where the sum ranges over
increasing trees $T$ on $\{0,1,2,\ldots,n\}$.
To expand the product on the right side of (\ref{x})
we may consider the power of $t_{n-1}$ first, then
the power of $t_{n-2}$, and so on. 
By the Pr\"ufer correspondence, 
a power $t_i^j$ in the expansion indicates that 
the vertex $i$ has outdegree $j$,
contributing the factor $g_j$ in the $A$-polynomial.
Thus, we arrive at the formula of Comtet, see
 Chen-Fu-Wang \textsuperscript{\ref{fn-2026}}.

For the case $g(x)=x$, 
the expansion of $(xD)^n$ can be viewed as a special case
of the normal ordering problem. As noted by 
Schork \cite{Schork2021}, the idea of
normal ordering goes back to Scherk \cite{Scherk-1823}. Consider the operators \(X\) and \(D\)
satisfying the relation
\begin{align}
\label{eq:weyl}
    DX-XD=I,
\end{align}
where \(I\) is the identity operator. 
The associative algebra generated by \(X\) and \(D\), subject
to \eqref{eq:weyl}, is called the first Weyl algebra.

Any monomial in \(X\) and  \(D\) can be written as
\[
\omega=X^{r_1}D^{s_1}X^{r_2}D^{s_2}\cdots X^{r_n}D^{s_n},
\]
where \(r_k,s_k \geq 0 \). The word 
\(\omega\) is said to be in normal ordered form if it is written as
\begin{align}
\label{no}
\omega=\sum_{r,s\in \mathbb{N}} S_{r,s}(\omega)X^rD^s,
\end{align}
where the coefficients \(S_{r,s}(\omega)\) are uniquely
determined and are called the normal ordering coefficients.

Ma-Mansour-Yeh-Yeh \cite{Ma-Mansour-Yeh-Yeh2024}
considered the normally ordered
expansion of $(w(x,y,\ldots)D_G)^n$, where
$D_G$ is the formal derivative of a grammar \(G\). 
For example, given a grammar \[ G=\{x \to u(x,y), \, y \to v(x,y)\}\]
and a function $w(x,y)$,
one may expand
$(w(x,y)D_G)^n$ as
\[
(w(x,y)D_G)^n
=
\sum_{k=0}^{n}\xi_{n,k}(x,y)w^k(x,y)D_G^k.
\]
The coefficients \(\xi_{n,k}(x,y)\) are
determined by the grammar $G$ and the weight function \(w(x,y)\).

Let \(\operatorname{cdes}(\pi)\) denote the
number of cycle descents of a permutation  \(\pi\).
A cycle descent of a permutation  is defined as
follows. In the following, we always write a permutation in standard cycle form, in which each cycle has its smallest element first and the cycles are written in increasing order of their first elements. If $j$ immediately 
follows \(i\) in this linear
representation and $i>j$, 
then we call $(i,j)$ a cycle descent.
For example,  the permutation 
  $\pi=(1,4,2)(3,5,7)(6,9,8)$ 
has two cycle descents \((4,2)\) and \((9,8)\). 
Ma-Mansour-Yeh-Yeh \cite{Ma-Mansour-Yeh-Yeh2024} 
obtained the following result.

\begin{theorem}[Ma-Mansour-Yeh-Yeh \cite{Ma-Mansour-Yeh-Yeh2024}]
Let
\[
G=\{x\to y,\; y\to py\}.
\]
For \(n\geq 1\), we have
\begin{align}
\label{normal-grammar}
(xD_G)^n
=
\sum_{\pi\in S_n}
x^{\,n-\exc(\pi)}
y^{\,\exc(\pi)}
p^{\,\operatorname{cdes}(\pi)}
D_G^{\,\cyc(\pi)}.
\end{align}
\end{theorem}
If we use a variable $q$ to replace the operator $D_G$ in
the above sum, we get multivariate polynomials called $(\operatorname{cdes},\operatorname{cyc}) (p,q)$-Eulerian polynomials. 
 Ma-Mansour-Yeh-Yeh \cite{Ma-Mansour-Yeh-Yeh2024} also showed, by means of grammars, that the Eulerian polynomials,
 the second-order Eulerian polynomials,
 the type \(B\)  Eulerian polynomials and the up-down run polynomials can all be represented by normally ordered expansions.

\section{The Ramanujan Polynomials}
 
As remarked by Berndt \cite{Berndt-I,  Berndt-etal},
no combinatorial perspective 
seems to be alluded to in the 
original definition of Ramanujan polynomials.
As time passes, we have come to see that
they legitimately belong to the  subject of
enumerative combinatorics.

The Ramanujan polynomials arise in three
seemingly unrelated contexts. First,
Ramanujan \cite{Berndt-I,  Ramanujan} defined
the polynomials $\psi_k(r, x)$ 
which we call the Ramanujan polynomials,    via the following
relation, 
\begin{align}\label{align:psi}
  \sum_{k=0}^\infty \frac{(x+k)^{r+k}\,e^{-u(x+k)}\,u^{k}}{k!}
  =\sum_{k=1}^{r+1}\frac{\psi_k(r, x)}{(1-u)^{r+k}}.
\end{align}

The second definition of the
Ramanujan polynomials has a different story.
As a refinement of Cayley's formula for
labeled trees, Shor \cite{Shor} came up with an insertion 
algorithm for labeled
rooted trees based on the notion of 
improper edges. Let $T$ be a rooted tree on $[n]$,
and $(i,  j)$ be an edge of $T$ with $j$ being 
a child of $i$. 
If $i$ is smaller than  any descendant of $j$
(including $j$), then $(i,j)$ is called a proper edge;
otherwise $(i,j)$ is called an improper edge. 
The algorithm leads to a 
recurrence relation for the number of 
rooted trees on $[n]$ with $k$ improper edges.
Shor then introduced the polynomials $Q_{n,k}(x)$, where
$x$ is treated as a positive integer. 

Zeng \cite{Zeng} gave an interpretation 
of $Q_{n,k}(x)$ by treating $x$ as an indeterminate
and observed a connection 
between Shor's polynomial and the Ramanujan polynomials
$\psi_k(r, x)$, namely
\begin{align} \label{Qnkx}
Q_{n, k}(x)=\psi_{k+1}(n-1, x+n). 
\end{align}
Let  ${\cal T}_{n+1, \, k}$ be
the set of rooted trees on $[n+1]$ with root $1$ and with
$k$ improper edges. Here is the interpretation of Zeng:
\begin{align}
Q_{n, \,  k}(x)=
\sum\limits_{T\in\mathcal{T}_{n+1, \, k}}x^{\deg_{T}(1)-1}, 
\label{zeng-1}
\end{align}
where $\deg_T(i)$ denotes the outdegree of
the vertex $i$ in  $T$. 

The third story about the Ramanujan is also surprising.
Around the same time of Shor's algorithm,
Dumont-Ramamonjisoa \cite{Dumont-Ramamonjisoa}
looked at a functional equation of Ramanujan:
\begin{align}\label{RE}
    x = y\, e^{-y} + \frac{a-1}{a} (e^{-y} - 1),
\end{align}
where $y$ is a series in $x$ without a constant term and $a$ is a parameter. 
To solve this equation in combinatorial 
terms, Dumont-Ramamonjisoa independently
developed the insertion algorithm as
provided by Shor. Moreover, as noted by 
Chen-Fu-Wang \cite{Chen-Fu-Wang-2025},
the functional equation
satisfied by the generating
function of the Ramanujan polynomials 
turns out to be the functional equation \eqref{RE} also due to Ramanujan.
This leaves one to wonder whether Ramanujan realized
that the polynomials $\psi_k(r,x)$ are somehow related to the 
functional equation \eqref{RE}. 
  
In order to derive
combinatorial expansions of
$y$ and $e^y$, Dumont-Ramamonjisoa started from the functional equation \eqref{RE} and derived the grammar 
\begin{align}
A \rightarrow A^3 S,  \; S \rightarrow AS^2.
\end{align}
 Then they devised an insertion algorithm 
 for rooted trees and showed that the
 recurrence relation for the enumeration
 of rooted tree with respect to the number
 of improper edges is consistent with 
 the recurrence relation governed by the grammar, 
 and so the functional equation was solved combinatorially.

There are two versions of the enumeration of rooted trees with
respect to the number of improper edges. 
Let $\mathcal{R}_{n, k}$ denote
the set of rooted trees on $[n]$ with  $k$ improper 
edges, and let $R(n, k)$ denote
the number of rooted trees in $\mathcal{R}_{n, k}$.
On the other hand, let $T(n, k)$ denote
the number of trees in $\mathcal{T}_{n+1, k}$, that is,
the number of rooted trees on 
$[n+1]$ with root $1$ and with $k$ improper edges. 
The numbers $R(n, k)$ and $T(n, k)$ 
are listed as sequences A054589 and
A217922 in OEIS \cite{OEIS}. 

Dumont-Ramamonjisoa found the following 
solution to the functional equation \eqref{RE}, 
along with the expansion of $e^y$,
\begin{align*}
    y &  = \sum_{n \ge 1} \left( \sum_{k=0}^{n-1}
      R(n, k)\, a^{n+k} \right) \frac{x^n}{n!},  \\[6pt]
      e^y & =1+ \sum_{n \ge 1} \left(\sum_{k=0}^{n-1}
      T(n, k)\, a^{n+k} \right) \frac{x^n}{n!}. 
\end{align*}

They also gave combinatorial interpretations of some
special cases of 
$Q_{n,\,k}(x)$,  such as $Q_{n, k}(0) = R(n, k)$,  $Q_{n, k}(1) = T(n, k)$. In addition,
they showed that $Q_{n, k}(-1)$ equals the
number of trees on $[n]$ with \(k\) improper edges in which vertex \(1\) is a leaf. It is worth mentioning that
Wang-Zhou \cite{Wang-Zhou} showed that
$Q_{n, k}(-1)$  is related to refined orbifold Euler characteristics of
the moduli space of stable curves of genus $0$ with $n$ marked points.
 
Based on the grammar of  Dumont-Ramamonjisoa,
Chen-Yang \cite{Chen-Yang-2021}
gave a grammar for the Ramanujan polynomials $Q_{n, k}(x)$.
By utilizing an extra label, Chen-Fu-Wang
\cite{Chen-Fu-Wang-2025} presented
a grammar for $Q_{n, k}(x)$ 
called a literal grammar, in the sense that 
the grammar is a direct translation of the
recursive construction. The idea of a literal
grammar not only provides a clue to 
build a grammar, it may also be informative
for specializations. Moreover, a literal grammar
might be more convenient to perform the 
grammatical calculus. 

The insertion algorithm of Shor and Dumont-Ramamonjisoa
leads to the following literal grammar:
\begin{align}\label{gram:Ramanujan}
 a\rightarrow axv, \; x\rightarrow xvz,  \;z\rightarrow vz^2,  \; v\rightarrow uv^2z,  \; 
     u\rightarrow u^2vz . 
\end{align}

Assume that $n\geq 1$ and
$T$ is a tree on $[n-1]$ with root $1$.
Consider the insertion of $n$ into $T$ to generate
a rooted tree $T'$ on $[n]$ with root 1.  
Now, the root $1$ is labeled by $a$,  
the children of the root
are labeled by $x$, the rest of the vertices
are labeled by $z$. All edges are labeled by
$v$, and an improper edge is endowed with
an extra label $u$. The reason to label
the children of the root is for the
sake of recording the outdegree of the root,
as required for the computation of $Q_{n,k}(x)$.

Depending on the outdegree of
$n$ in $T'$,  we encounter three cases,  
namely, the outdegree is $0$, or  $1$,  or 
greater than $1$. The three cases can be 
identified with three labels $z$,  $v$ and $u$. 
Accordingly, the insertions in the three cases
are called the leaf-insertion,  the $v$-insertion and
the $u$-insertion.

\begin{itemize}

\item[1.] The leaf-insertion. First, if $n$ is inserted as
child of the root, we get the rule
\[ a \rightarrow axv. \]
Second, if $n$ is inserted as a child of a 
vertex labeled by $x$, we get the rule
\[ x \rightarrow x v z. \]
Third, if $n$ is inserted as a child of a vertex $i$ 
labeled by $z$, that is, $i$ is neither the root nor 
a child of the root, we get the rule
$$z \to zvz, $$
as illustrated in 
Figure \ref{fig:2}.

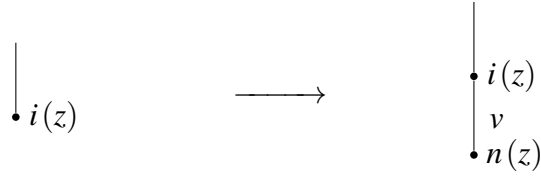
\begin{figure}[ht]
\centering
\begin{minipage}{0.3\textwidth}
\centering
\begin{tikzpicture}[scale=0.8, 
  level 1/.style = { level distance   = 13mm, 
                     sibling distance = 36mm }, 
  level 2/.style = { level distance   = 13mm, 
                     sibling distance = 20mm }, 
  double line/.style={double,  double distance=1pt,  line width=0.18mm}
  ]
  
\node (root) {} 
    child {node [ns, label=0:{$i \,  (z)$}](i){}
    };
   
\end{tikzpicture}
\end{minipage}%
\begin{minipage}{0.1\textwidth}
\centering
$\displaystyle \xrightarrow{\hspace*{1cm}}$ 
\end{minipage}%
\begin{minipage}{0.3\textwidth}
\centering
\begin{tikzpicture}[scale=0.8, 
  level 1/.style = { level distance   = 13mm, 
                     sibling distance = 36mm }, 
  level 2/.style = { level distance   = 13mm, 
                     sibling distance = 36mm },  
  level 3/.style = { level distance   = 13mm, 
                     sibling distance = 10mm }, 
  double line/.style={double,  double distance=1pt,  line width=0.18mm}]
  
\node (root) {} 
    child {node [ns, label=0:{$i\, (z)$}](n){}
    child {node [ns, label=0:{$n \, (z)$}](j){}
    }
    };
    
    \node [right=3mm,  above=2.5mm] at (j){$v$};
\end{tikzpicture}
\end{minipage}
\caption{A leaf-insertion.}
\label{fig:2}
\end{figure}

\item[2.] A $v$-insertion is illustrated in 
Figure \ref{fig:3}. We get the rule 
$$v \to v(uvz).$$

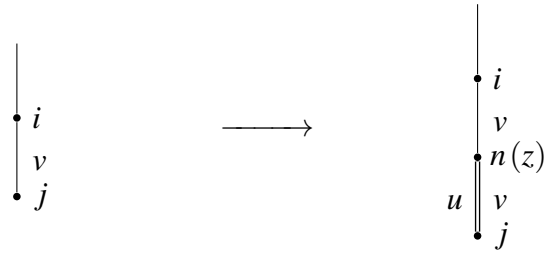
\begin{figure}[ht]
\centering
\begin{minipage}{0.3\textwidth}
\centering
\begin{tikzpicture}[scale=0.8, 
  level 1/.style = { level distance   = 13mm, 
                     sibling distance = 36mm }, 
  level 2/.style = { level distance   = 13mm, 
                     sibling distance = 20mm }, 
  double line/.style={double,  double distance=1pt,  line width=0.18mm}
  ]
  
\node (root) {} 
    child {node [ns, label=0:$i$](i){}[grow=down]
    child {node [ns, label=0:{$j$}](j){}
    }
    };
    
    \node [right=3mm,  above=2.5mm] at (j){$v$};
\end{tikzpicture}
\end{minipage}%
\begin{minipage}{0.1\textwidth}
\centering
$\displaystyle \xrightarrow{\hspace*{1cm}}$ 
\end{minipage}%
\begin{minipage}{0.3\textwidth}
\centering
\begin{tikzpicture}[scale=0.8, 
  level 1/.style = { level distance   = 13mm, 
                     sibling distance = 36mm }, 
  level 2/.style = { level distance   = 13mm, 
                     sibling distance = 36mm },  
  level 3/.style = { level distance   = 13mm, 
                     sibling distance = 10mm }, 
  double line/.style={double,  double distance=1pt,  line width=0.18mm}]
  
\node (root) {} 
    child {node [ns, label=0:$i$](i){}[grow=down]
    child {node [ns, label=0:{$n\, (z)$}](n){}
    child {node [ns, label=0:{$j$}](j){}}
    }
    };
    
    \node [right=3mm,  above=2.5mm] at (n){$v$};
    \node [right=3mm,  above=2.5mm] at (j){$v$};
    \node [left=3mm,  above=2.5mm] at (j){$u$};
    \draw[double line] (n)--(j);
\end{tikzpicture}
\end{minipage}
\caption{A $v$-insertion.}
\label{fig:3}
\end{figure}

\item[3.] A $u$-insertion is associated with an improper
edge with  a label $u$. In this case, a new improper 
edge is created,  see Figure \ref{fig:4}. We get the rule 
$$u \to u(uvz).$$
where we use $\beta(j)$ to denote
the minimum vertex among all the 
descendants of $j$,  and the vertices
$j_1, j_2, \ldots, j_l$ are
arranged subject to the
condition $\beta(j_1) < \beta(j_2) < \cdots < \beta(j_l)$.
Note that $\{j_{d+1}, \ldots, j_l\}$ is allowed to be empty.

\begin{figure}[ht]
\centering
\begin{minipage}{0.3\textwidth}
\centering
\begin{tikzpicture}[scale=0.7, 
  level 1/.style = { level distance   = 13mm, 
                     sibling distance = 20mm }, 
  level 2/.style = { level distance   = 13mm, 
                     sibling distance = 20mm }, 
  double line/.style={double,  double distance=1pt,  line width=0.18mm}
  ]

\node (root) {} 
    child {node [ns, label=30:{$i$}](i){}
        child {node [ns, label=270:{$j_1$}](1){}}
        child {node [ns, label=270:{$j_d$}](d){}}
        child {node [ns, label=270:{$j_{d+1}$}](s){}}
        child {node [ns,  label={[label distance=4.5pt]270:{ $j_l$}}](l){}
        }
    };
    \node [right=4mm] at (1){$\ldots$};
    \node [right=3.5mm] at (s){$\ldots$};
    \node [left=-1mm,  above=2.5mm] at (d){$u$};
    \draw[double line] (i) -- (1);
    \draw[double line] (i) -- (s);
    \draw[double line] (i) -- (d);
    \draw (i) -- (l);
    \draw (root) -- (i);
    \coordinate (left) at ([xshift=-0.55cm, yshift=0.22cm] s);
    \coordinate (right) at ([xshift=0.5cm, yshift=-0.19cm] l);
    \draw (left) rectangle (right);
\end{tikzpicture}
\end{minipage}%
\begin{minipage}{0.1\textwidth}
\centering
$\displaystyle \xrightarrow{\hspace*{0.5cm}}$ 
\end{minipage}%
\begin{minipage}{0.3\textwidth}
\centering
\begin{tikzpicture}[scale=0.7, 
  level 1/.style = { level distance   = 13mm, 
                     sibling distance = 20mm }, 
  level 2/.style = { level distance   = 13mm, 
                     sibling distance = 20mm },  
  double line/.style={double,  double distance=1pt,  line width=0.18mm}]
  
\node (root) {} 
    child {node[ns, label={[yshift=2pt]0:{$n \,  (z)$}}](n){}[grow=down]
        child {node [ns, label=270:{$j_1$}](1){}}
        child {node [ns, label=270:{$j_d$}](d){}}
        child {node [ns, label=0:{$i$}](i){}
        child {node [ns, label=270:{$j_{d+1}$}](s){}}
        child {node [ns, label={[label distance=4.5pt]270:{$j_{l}$}}](l){}}
        }
    };
    
    \node [right=5mm] at (1){$\ldots$};
    \node [right=0.5mm] at (s){$\ldots$};
    \node [left=9mm,  above=1mm] at (i){$u$};
    \node [right=-3mm,  above=2.5mm] at (i){$v$};
    \node [left=2mm,  above=2.5mm] at (d){$u$};
    \draw[double line] (n) -- (1);
    \draw[double line] (n) -- (i);
    \draw[double line] (n) -- (d);
    \draw[double line] (i) -- (s);
    \draw (root) -- (n);
    \coordinate (left) at ([xshift=-0.4cm, yshift=0.15cm] s);
    \coordinate (right) at ([xshift=0.4cm, yshift=-0.15cm] l);
    \draw (left) rectangle (right);
\end{tikzpicture}
\end{minipage}%
\caption{A $u$-insertion.}
\label{fig:4}
\end{figure}
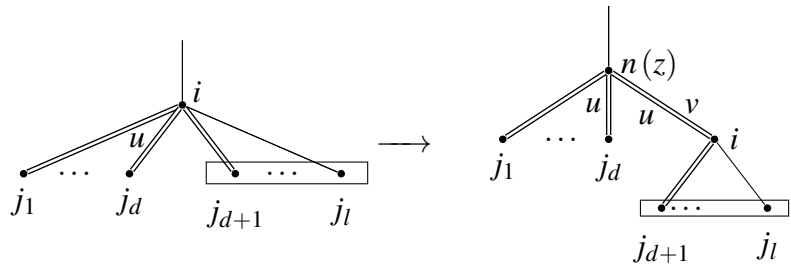
\end{itemize}

Let $D$ be the formal derivative
of the grammar \eqref{gram:Ramanujan}. 
We have the following grammatical representation of 
$Q_{n,k}(x)$
\begin{align}
    D^n(a)=axv^nz^{n-1}\sum_{k=0}^{n-1} Q_{n, k}(xz^{-1})\, u^k. \label{Dna}
\end{align}
The grammatical calculus
can be employed to derive
the functional equation 
for  the generating function 
of the Ramanujan polynomials. Moreover,
it can be used to justify a grammatical
identity bearing a combinatorial interpretation.

Let
\begin{align*}
    R_n(u)   =  \sum_{k=0}^{n-1}
      R(n, k) \, u^k.
\end{align*}
The first few values of $R_n(u)$ are given below: 
\begin{align*}
R_1(u)&= 1, \\[3pt]
R_2(u)&= 1+ u,\\[3pt]
R_3(u) &=2+4u+ 3u^2, \\[3pt]
R_4(u) &=6+18u+25u^2+15u^3, \\[3pt]
R_5(u) & = 24+96 u+190 u^{2}+210 u^{3} +105 u^{4}. 
 \end{align*}
 
In the study of the monotone properties of 
polynomials, we often use $q$ as the variable. 
Let $\{f_n(q)\}_{n\ge 0}$ be a sequence of 
polynomials in $q$. By writing $f(q) \ge_q g(q)$ we
mean $f(q)-g(q)$ is a polynomial in $q$ 
with nonnegative coefficients. 
If   
\[
f_{m+1}(q)\, f_{m-1}(q) \ge_q f_m(q)^2, 
\]for all $m \ge 1$,
then we say that $\{f_n(q)\}_{n\ge 0}$ is $q$-log-convex. 
If  
\[
f_{m-1}(q)\, f_{n+1}(q) \ge_q f_m(q)\, f_n(q), 
\]
for all $n \ge m \ge 1$, 
then we say that 
$\{f_n(q)\}_{n\ge 0}$ is strongly
$q$-log-convex. 

Chen-Wang-Yang \cite{Chen-Wang-Yang-2011} 
showed that 
$\{R_{n+1}(q)\}_{n\ge 0}$ is strongly $q$-log-convex. 
Sokal \cite{Sokal-2021} obtained a more
general result, namely, every minor of
the Hankel matrix  $(R_{i+j+1}(q))_{i, j\ge0}$
is a polynomial with nonnegative coefficients.

\section{The Insertion Algorithm for Plane Trees}

We observe that the insertion algorithm of 
Shor and Dumont-Ramamonjisoa applies to
labeled plane trees without 
any additional effort. The key point is to treat
every edge as an improper edge. In other words,
we may ignore the notion of improper edges while
recursively constructing plane trees by insertions.
  
Let $n\geq 1$, and let $\mathcal{P}_n$ denote the 
set of labeled plane trees on $[n+1]$.
It is clear that  $\lvert \mathcal{P}_n \rvert =
(n+1)! C_n$, where  $C_n$ is the $n$-th Catalan number, 
that is, the number of unlabeled
plane trees with $n$ edges, 
see sequence A000108 of 
OEIS \cite{OEIS}. 
Let $\mathcal{O}_{n}$ be the set of
plane trees on $\{0, 1, 2, \ldots, n\}$ with root $0$.
It is clear that  $\lvert \mathcal{O}_n \rvert =
n! C_n$.

The ballot numbers  $T(n,  k)$
can be viewed as a refinement of the
Catalan numbers, see sequence A009766 in OEIS  \cite{OEIS}.
For $n \ge 0$ and $0\le k \le n$, the ballot numbers \(T(n,k)\) are defined by
\[ T(n, k)= \frac{n-k+1}{n+1} \binom{n+k}{k},\]
which satisfy the recurrence relation
\begin{equation}\label{ballot-recurrence}
(n+2)\,T(n+1, k) = (n-k+2)\, T(n, k) + (2n+2k)\,T(n, k-1).
\end{equation}

The number $T(n,k)$ can be interpreted 
as the number of unlabeled plane trees with 
$n+1$ edges such that the 
outdegree of the root equals $n-k+1$. 
Making use of the recurrence relation (\ref{ballot-recurrence}),
Ma \cite{Ma-2013-b} found a grammar to generate
$T(n, k)$.

 \begin{theorem}[Ma \cite{Ma-2013-b}]
Let $D$ be the formal derivative of the grammar
     \begin{equation}\label{Ma-Ballot}
 G = \{ a \rightarrow a^2b^2, ~ b \rightarrow b^3c^2,  ~ c \rightarrow b^2c^3 \}. 
 \end{equation}
 For  $n \geq 0$, we have
 \begin{equation}
 \label{grammar:Ballot}
 D^n(a^2b^2) = (n+1)! a^2 b^{2n+2} \sum_{k=0}^{n} T(n, k) a^{n-k} c^{2k}.
 \end{equation}
 \end{theorem}

The Narayana numbers can also be generated
by a grammar. 
For $n \ge 1, 1 \le k \le n$,
the Narayana numbers, see sequence A001263 in OEIS \cite{OEIS}, are defined by 
\[ N(n,  k)=\frac{1}{n}\binom{n}{k}\binom{n}{k-1}. \]
They are also a refinement of the Catalan numbers and
satisfy the following recurrence relation
\begin{equation}\label{Nara-recurrence}
    (n+2)\, N(n\!+\!1, k)=(n+2k)\, N(n, k)+(3n+4-2k)\, N(n, k-1).
\end{equation} 

It is well known that 
$N(n,k)$ equals the number of 
unlabeled rooted plane trees with $n+1$ vertices
and $k$ leaves. 
Employing the
recurrence relation (\ref{Nara-recurrence}), Ma-Ma-Yeh \cite{Ma-Ma-Yeh} found a grammar for the
Narayana polynomial $N_n(x)$ and showed that
they are $\gamma$-positive, where
 \begin{align*}
    N_n(x)=\sum_{k=1}^{n}N(n,k)x^k.
\end{align*}

\begin{theorem}[Ma-Ma-Yeh \cite{Ma-Ma-Yeh}] \label{Ma-Theorem}
    Let $D$ be the formal derivative of the grammar  
    \begin{align}\label{Ma-Ma-Yeh-gram}
        G=\{u\rightarrow u^2v^3,  ~v\rightarrow u^3v^2\}.  
    \end{align}
For $n\ge1$, we have
    \begin{align}
        D^n(u^2) = (n\!+\!1)! \, \sum_{k=1}^{n} N(n,  k) \, u^{3n-2k+2}\, v^{n+2k}. 
    \end{align}
Moreover, $N_n(x)$ is $\gamma$-positive for all $n$. 
\end{theorem}

As noted by Ma-Ma-Yeh \cite{Ma-Ma-Yeh},  
the $\gamma$-positivity of $N_n(x)$ can be deduced
from the change of variables: $x=uv,  y=u^2+v^2$ and $z=u^2$. 
Ma raised the question
of finding a grammatical labeling
to justify the grammar  (\ref{Ma-Ma-Yeh-gram}) for Theorem \ref{Ma-Theorem}, see  \cite{Lin-Liu-Wang-Zang-2024}. 
Lin-Liu-Wang-Zang \cite{Lin-Liu-Wang-Zang-2024} and Yang-Zhang \cite{Yang-Zhang-2025} independently obtain
grammatical labelings, thereby answering Ma's question. The grammar given by Yang-Zhang is as follows,   
\begin{equation} \label{Yang-Zhang-gram}
 G=\{t\rightarrow t^2(x+y),  ~x\rightarrow 2txy,  ~ y\rightarrow 2txy\}.
\end{equation}
Let $D$ be the formal derivative
of the grammar \eqref{Yang-Zhang-gram}. For $n\geq 1$,  we have 
    \begin{align}
        D^{n}(y) =(n+1)!t^{n} \sum_{k=1}^n
      N(n,  k)x^{k}y^{n-k+1}.
    \end{align}
Set $t = uv$,    $x = u^2$ and  $y = v^2$,    the grammar (\ref{Ma-Ma-Yeh-gram})  is transformed into (\ref{Yang-Zhang-gram}).  Setting $y=x$, 
 the grammar (\ref{Yang-Zhang-gram}) reduces to the
 grammar for the Catalan numbers
    \begin{align}\label{Yang-Zhang-Catalan}
    G=\{t\rightarrow 2t^2x,  ~ x\rightarrow 2tx^2\}.
\end{align}
 For $n\geq 0$,  we have 
    \begin{align}
        D^{n}(x)= (n\!+\!1)!\,   t^{n}\,   x^{n+1}\,   C_{n}.
    \end{align}

 Deutsch \cite{Deutsch-1999} obtained a formula for Dyck paths, 
 which can be viewed as a unification
 of the formulas for the Narayana numbers
 and the ballot numbers. We call the involved 
 numbers the Deutsch numbers, denoted
 $D_r(n,  k)$. Combinatorially, \(D_r(n,k)\) is the number of unlabeled plane trees with $n+1$ vertices with root outdegree $r$ and $k$ leaves.
 
\begin{theorem}[Deutsch \cite{Deutsch-1999}]
    For $n\geq 1$ and  $1\leq r, k\leq n$,   
    \begin{align} \label{DF}
        D_r(n,  k) =
 \frac{r}{k}\binom{n-1}{k-1}\binom{n-1-r}{k-r} .
    \end{align}
\end{theorem}

Next we proceed to adapt the insertion algorithm of Shor and Dumont-Ramamonjisoa to derive a grammar for the
Deutsch numbers, which specializes to grammars 
for the Catalan numbers and the Narayana numbers.
For a plane tree in ${\mathcal O}_n$, define 
a labeling scheme as follows. 
\begin{itemize}
    \item A label $r$ signifies
    a position where one can insert a new vertex
as a child of the root.

\item A label $s$ signifies a position where one can 
insert a new vertex as a child of a vertex that is neither
the root nor a leaf.

\item A label $t$ signifies the position where one
can insert a new vertex as a child of a leaf.

\item Assign a label $v$ to every edge.

\item In addition, for any nonroot internal vertex $i$, 
label the edge from $i$ to its last child by $w$.
Label the edge  from $i$ to its other children by
$u$. 
\end{itemize}
Figure \ref{COplanetreeexam.fig} gives an illustration.
The weight of a tree
in ${\mathcal O}_n$
is defined to be the product of the labels.
Let $G_{n}(r,  s,  t; u,  v,  w)$
denote the total weight of the
trees in $\mathcal{O}_n$.

\begin{figure}
\begin{center}
    \begin{tikzpicture}
 [vertex/.style={shape=circle,   draw,   inner sep=1pt,   fill=black},  
 subtree/.style={shape=ellipse,   draw,  minimum width=1.5cm,   minimum height=.5cm},  
 every fit/.style={ellipse,  draw,  inner sep=-2pt},  
sibling distance=2.5cm,  level distance=16mm,  
 leaf/.style={label={[name=#1]below:$ $}},  scale=0.7]
 
\node[vertex,  label=30:{$0$}]{}[grow=down]
child {node [vertex,   label=-90:{$r$}]{}edge from parent [dashed]}
child {node [vertex,  label=0:{$5$}]{}child {node [vertex,   label=-90:{$t$}]{}edge from parent [dashed]}
edge from parent [solid] node[right,  yshift=-3.5pt]{$v$}}
child {node [vertex,   label=-90:{$r$}]{}edge from parent [dashed]}
child {node [vertex,  label=10:{$3$}]{}[sibling distance=2cm,  level distance=16mm]
child {node [vertex,   label=-90:{$s$}]{}edge from parent [dashed]}
  child {node [vertex,   label=0:{$2$}]{}[sibling distance=2cm,  level distance=16mm]
  child {node [vertex,   label=-90:{$s$}]{}edge from parent [dashed]}
  child {node [vertex,   label=0:{$1$}]{}
   child {node [vertex,   label=-90:{$t$}]{}edge from parent [dashed]}
  edge from parent [solid] node[right]{$vw$}}
  child {node [vertex,   label=-90:{$s$}]{}edge from parent [dashed]} edge from parent node[right,  yshift=-4.5pt]{$uv$}}
  child {node [vertex,   label=-90:{$s$}]{}edge from parent [dashed]}
 child{node[vertex,   label=0:{$4$}]{}
 child {node [vertex,   label=-90:{$t$}]{}edge from parent [dashed]}
 edge from parent node[left, yshift=-4.5pt,  xshift=-0.3pt]{$vw$} }
 child {node [vertex,   label=-90:{$s$}]{}edge from parent [dashed]}edge from parent [solid] node[left,  yshift=-3.5pt]{$v$}} 
 child {node [vertex,   label=-90:{$r$}]{}edge from parent [dashed]};
\end{tikzpicture}
 \end{center}
\caption{The labeling of a plane tree in $\mathcal{O}_{5}$}
\label{COplanetreeexam.fig}
 \end{figure}
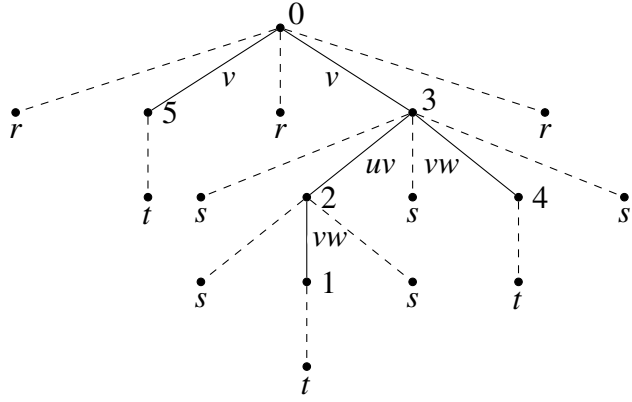 

To read off the insertion algorithm in accordance 
with the above labeling scheme, it is straightforward to
compose a grammar, with special attention only to the
insertion associated with an edge labeled by $w$,
which we call a $w$-insertion, as shown in Figure 
 \ref{w-insertion}. Notice that a new leaf is created
 by a $w$-insertion. So we get the rule 
$$w \to stuvw.$$

\begin{figure}[ht]
\centering
\begin{minipage}{0.3\textwidth}
\centering
\begin{tikzpicture}[scale=0.7, 
  level 1/.style = { level distance   = 13mm, 
                     sibling distance = 20mm }, 
  level 2/.style = { level distance   = 13mm, 
                     sibling distance = 20mm }, 
  double line/.style={double,  double distance=1pt,  line width=0.18mm}, 
  ]

\node(root){}
    child{ node [ns, label=30:{$i$}](i){}[grow=down]
        child {node [ns, label=270:{$j_1$}](1){}}
        child {node [ns, label=270:{$j_d$}](d){}
        }
    };
    \node [right=4mm] at (1){$\ldots$};
    \node [right=0.511mm,  above=2.7mm] at (d){$vw$};
    \draw[double line] (i) -- (1);
    \draw[double line] (i) -- (d);
    \draw (root) -- (i);
\end{tikzpicture}
\end{minipage}%
\begin{minipage}{0.1\textwidth}
\centering
$\displaystyle \xrightarrow{\hspace*{0.5cm}}$ 
\end{minipage}%
\begin{minipage}{0.3\textwidth}
\centering
\begin{tikzpicture}[scale=0.7, 
  level 1/.style = { level distance   = 13mm, 
                     sibling distance = 20mm }, 
  level 2/.style = { level distance   = 13mm, 
                     sibling distance = 20mm },  
  double line/.style={double,  double distance=1pt,  line width=0.18mm}]
  
\node{}
    child{
    node [ns, label={[yshift=3pt]0:{$n$}}](n){}[grow=down]
            child {node [ns, label=270:{$j_1$}](1){}}
            child {node [ns, label=270:{$j_d$}](d){}}
            child {node [ns,  label={[label distance=2.5pt]270:{$i$}}](i){}
            }
    };
    
    \node [right=5mm] at (1){$\ldots$};
    \node [right=2.5mm,  above=1.6mm] at (d){$vu$};
    \node [left=3.5mm,  above=2.7mm] at (i){$vw$};
    \draw[double line] (n) -- (1);
    \draw[double line] (n) -- (i);
    \draw[double line] (n) -- (d);
\end{tikzpicture}
\end{minipage}%

\caption{A $w$-insertion.}
\label{w-insertion}
\end{figure}
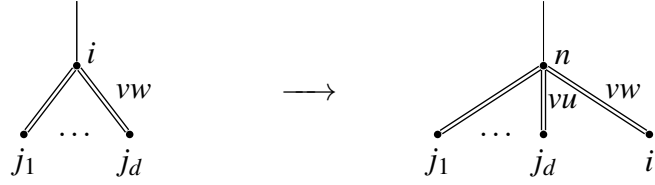

\begin{theorem}
Let $D$ be the formal derivative of the grammar   
 \begin{align}\label{chen-gram}
    G=\{ r\rightarrow r^2tv,  ~ s\rightarrow s^2tuv,  ~t\rightarrow s^2tvw,  ~u\rightarrow s^2uvw,  ~v\rightarrow s^2v^2w,  ~w\rightarrow stuvw\}.
\end{align}
For $n\geq 0$, we have   
\begin{align}\label{chen-gram-2-aim}
    D^{n}(r)=G_{n}(r,  s,  t; u,  v,  w).
\end{align}
\end{theorem}

One may compute the generating function of $D^n(r)$
by the grammatical calculus. Examples of using the
grammatical calculus to compute generating functions
can be found in 
\cite{Chen-Fu-2017,  
Chen-Fu-2023-JACO,  Chen-Fu-2023-AC,  Chen-Fu-Wang-2025,  Dong-etal}.

Setting $r=a$, $s=u=c$ and $t=v=w=b$, 
the grammar \eqref{chen-gram} reduces to the grammar
\eqref{Ma-Ballot}.
If we replace  $u,t,v,w$ by $s$ in \eqref{chen-gram}, then
we get the grammar  
\begin{align}
    G=\{r\rightarrow r^2s^2,  ~s\rightarrow s^5\},
\end{align}
which is a grammar for the ballot numbers. 
Let $D$ be the formal derivative of the above
grammar. Indeed, for $n\geq 0$, we have
\[ D^{n+1}(r) \big|_{s=1} =(n+1)!\sum_{k=0}^n T(n,  k) r^{n-k+2}. \]

\section{Grammars and Stable Polynomials}
\label{sec:9}
In this section, we demonstrate 
how grammars along with labeling schemes 
can play a role in constructing
stable multivariate polynomials associated with
certain combinatorial structures. 

A multivariate polynomial 
$f(z_1,  z_2,  \dots,  z_n) \in \mathbb{C}[z_1,  z_2,  \dots,  z_n]$ is called stable if 
\[ f(z_1,  z_2,  \dots,  z_n) \neq 0\]
as long as $\Ima(z_i) > 0$ for $(1\leq i \leq n)$.
A stable multivariate polynomial with real coefficients
is said to be real stable. 
A univariate polynomial $f(z) \in \mathbb{R}[z]$ 
with real coefficients is stable if and
only if it is real-rooted. 

Many univariate combinatorial polynomials are real-rooted,
while their multivariate refinements are stable. We now
consider a family of such multivariate combinatorial
polynomials. In answer to the two questions of Haglund-Visontai \cite{Haglund-Visontai} on stable
multivariate refinements of the second-order 
Eulerian polynomials, 
Chen-Hao-Yang \cite{Chen-Hao-Yang-2020} 
obtained grammars to generate
the Legendre-Stirling permutations and marked 
Stirling permutations.

For $n\geq 1$, let $[n]_2$ denote the multiset
$\{1,  1,  2,  2,  \dots,  n,  n\}$. Let $\pi = \pi_1\pi_2 \cdots \pi_{2n-1}\pi_{2n}$ be a permutation of 
$[n]_2$. We say that $\pi$ is a Stirling permutation
if for any $1\leq i \leq n$,  the elements between
the two occurrences of $i$, if any, are greater than $i$. 
Moreover, we assume that $\pi_0 = \pi_{2n+1} = 0$. 
For $n\geq 1$, let $Q_n$ denote the set of Stirling 
permutations of $[n]_2$.

For the purpose of producing stable multivariate 
polynomials related to Stirling permutations, 
we introduce the notion of marked Stirling permutations,
see Chen-Hao-Yang \cite{Chen-Hao-Yang-2020}. 
Let $\pi = \pi_1 \pi_2 \cdots \pi_{2n}$ be a Stirling
permutation of $[n]_2$. If  $\pi_i$ is the second
occurrence of an element $k$ and $\pi_i < \pi_{i+1}$,
then we may mark the element $\pi_i$ with a bar. Note that
it is also eligible not to mark this element.
A marked Stirling permutation is derived from
a Stirling permutation by
adding some bars.
Let $\bar{Q}_n$ denote the set of marked
Stirling permutations of $[n]_2$. 
For example, for $n=2$, there are four
marked Stirling permutations: 
\begin{align*}
    2211,  \quad 1221,  \quad 1122,  \quad 1\bar{1}22.
\end{align*}

Egge \cite{Egge} 
introduced the notion of 
Legendre-Stirling permutations. 
For $n \ge 1$, let $M_n$ denote the
multiset $\{1,  1,  \bar{1},  2,  2,  \bar{2},  \dots,  n,  n,  \bar{n}\}$. 
Let  $\pi = \pi_1 \pi_2 \dots \pi_{3n}$ be a permutation
of $M_n$. 
We say that $\pi$ is a Legendre-Stirling permutation
if for any $1\leq i \leq n$,
the elements between the two occurrences of the unbarred $i$, if
any, are greater than $i$. Note that the bars are
ignored as far as the comparison is concerned. 
We also assume that $\pi_0 = \pi_{3n+1} = 0$.

We first define the multivariate polynomial associated
with Legendre-Stirling permutations. Let $X = (x_1,  x_2,  \dots,  x_n)$, $Y = (y_1,  y_2,  \dots,  y_n)$,  $Z = (z_1,  z_2,  \dots,  z_n)$,  $U = (u_1,  u_2,  \dots,  u_n )$  and $V = (v_1,  v_2,  \dots,  v_n )$. Define  
\begin{align*}
    B_n(X,  Y,  Z,  U,  V) =  \sum_{\pi} \prod_{i \in X(\pi)} x_{\pi_i} \prod_{i \in Y(\pi)} y_{\pi_i} \prod_{i \in Z(\pi)} z_{\pi_i} \prod_{i \in U(\pi)} u_{\pi_i} \prod_{i \in V(\pi)} v_{\pi_i}, 
\end{align*}
where $\pi$ ranges over  Legendre-Stirling permutations
of $M_n$, $X(\pi),  Y(\pi),  Z(\pi),  U(\pi)$ and $V(\pi)$ 
are defined by 
\begin{align*}
X(\pi) &= \{i \mid \pi_{i-1} \le \pi_i,  \;  \pi_i 
\text{ is unbarred and is the first occurrence}\},  \\
Y(\pi) &= \{i \mid \pi_i > \pi_{i+1},  \;  \pi_i \text{ is unbarred}\},  \\
Z(\pi) &= \{i \mid \pi_{i-1} \le \pi_i,  \; \pi_i \text{ is
unbarred and is the second occurrence}\},  \\
U(\pi) &= \{i \mid \pi_{i-1} \le \pi_i,  \; 
\pi_i \text{ is barred}\},  \\
V(\pi) &= \{i \mid \pi_i > \pi_{i+1},  \; \pi_i \text{ is barred}\}.
\end{align*}

The polynomials \(B_n(X,Y,Z,U,V)\) 
can be generated by a grammar along with a labeling scheme.
 Using the grammar representation, 
 the stability of the multivariate polynomials
 follows from the differential operator
 with respect to the grammar preserves the 
 stability.  Here we need a theorem 
 of Borcea-Br\"and\'en \cite{Borcea-Branden-2009}
 on linear operators that preserve the stability of
 multiaffine polynomials. 

A polynomial $f(z_1,  z_2,  \dots,  z_n)$  is called
multiaffine if the degree of any variable $z_i$ is
at most one.  Let $\mathbb{C}^{\rm ma}[z_1,  z_2,  \dots,  z_n]$ denote the vector space of 
multiaffine polynomials in $z_1, z_2, \ldots, z_n$ over
complex numbers. A linear operator $T$ is called  stability preserving
if for any stable polynomial
$f \in \mathbb{C}^{\rm ma}[z_1,  z_2,  \dots,  z_n]$, $T(f)$ is
either stable or identically zero.

\begin{theorem}[Borcea-Br\"and\'en \cite{Borcea-Branden-2009}]
    Let $T$ be a linear operator from  $\mathbb{C}^{\rm ma}[z_1,  z_2,  \dots,  z_n]$ to $ \mathbb{C}[z_1,  z_2,  \dots,  z_n]$. If the polynomial
    $$T\left( \prod_{i=1}^n (z_i + w_i) \right) \in \mathbb{C}[z_1,  \dots,  z_n,  w_1,  \dots,  w_n]$$
    is stable, then $T$ is stability preserving on 
    multiaffine polynomials. 
\end{theorem} 

Using the above criterion and the grammatical representation,
the stability of \(B_n(X,  Y,  Z,  U,  V)\) can be deduced. 

\begin{theorem}[Chen-Hao-Yang \cite{Chen-Hao-Yang-2020}]
For \(n \ge 1\), the polynomial \(B_n(X,  Y,  Z,  U,  V)\) 
is stable. 
\end{theorem}

Let  $\pi = \pi_1 \pi_2 \cdots \pi_{2n}$ be a Stirling
permutation of $[n]_2$. 
Assume that $\pi_0=\pi_{2n+1}=0$. Define
\begin{align*}
A(\pi) & 
= \{i \mid \pi_{i-1} < \pi_i, \ 1 \le i \le 2n\}, \\[3pt]
D(\pi) &  = \{i \mid \pi_i > \pi_{i+1}, \ 1 \le i \le 2n\}, \\[3pt]
P(\pi) & = \{i \mid \pi_{i-1} = \pi_i, \ 1 \le i \le 2n\}. 
\end{align*}
We now come to the multivariate polynomials in connection
with marked Stirling permutations. For a marked Stirling permutation $\bar{\pi}$, 
we define the sets  
$D(\bar{\pi})$, $A(\bar{\pi})$ and $P(\bar{\pi})$
as if the bars are invisible, 
and we still assume that
$\pi_0=\pi_{2n+1}=0$. For example, for
$\bar{\pi}=1\;\bar{1}\;2\;2$, we have
\begin{align*}
    A(\bar{\pi})=\{1,3\},\quad
    D(\bar{\pi})=\{4\},\quad
    P(\bar{\pi})=\{2,4\}.
\end{align*}

We next define the multivariate polynomial associated
with marked Stirling permutations.
Define the multivariate polynomials
$$T_n(X, Y, Z) = \sum_{\bar{\pi}} \prod_{i \in A(\bar{\pi})} x_{\pi_i} \prod_{i \in D(\bar{\pi})} y_{\pi_i} \prod_{i \in P(\bar{\pi})} z_{\pi_i},$$
where $\bar{\pi}$ ranges over 
marked Stirling permutations of $[n]_2$.

\begin{theorem}[Chen-Hao-Yang \cite{Chen-Hao-Yang-2020}]
For \(n \ge 1\),  \(T_n(X,  Y,  Z)\) is stable. 
\end{theorem}

Recently, using the grammatical approach 
Yang-Zhang \cite{Yang-Zhang-2025} obtained
a multivariate version of the 
Narayana polynomials and proved their stability. 

\section{Noncommutative Grammars and $q$-Grammars}

In this final section, we discuss two recent developments. 

First, the Leibniz formula holds for commutative variables.
But the context-free property \eqref{prop:context-free} 
is concerned with the 
characterization of a formal language in terms of
production rules regardless of their contexts.
In theory, we do not have to require that
the variables be commutative. 
So the problem is how to 
employ appropriate operators in a noncommutative setting
to carry out an effective grammatical calculus.

Ehrenborg \cite{Ehrenborg2024} 
took the grammar of Dumont for the
Eulerian polynomials to the noncommutative platform. 
Assume that $a$ and $b$ are noncommutative variables.
Consider the 
substitution rules 
\[
a \to ba,\; b \to ab.
\]
This noncommutative grammar is closely related to the
\(ab\)-index and the \(cd\)-index of Eulerian posets.
By studying the behavior of powers of the pyramid and prism operators on a product, Ehrenborg derived exponential generating functions for the $cd$-indices of the $n$-dimensional simplices and the $n$-dimensional cubes, thereby providing a new perspective on the theory of $cd$-indices in polyhedral combinatorics.

The second development is the $q$-grammatical calculus. 
Enumeration problems involving Mahonian statistics often lead to \(q\)-binomial coefficients, 
  also called the Gaussian polynomials. 
A fundamental question is how to build a grammatical calculus with a suitable version of the $q$-Leibniz formula. 
Han-Ji-Xiong \cite{Han-Ji-Xiong2026} 
introduced a \(q\)-analogue of the grammatical calculus,
called the \(q\)-derivative grammar, or
\(q\)-grammar for short. Specifically, for a special
class of $q$-grammars (termed $q$-linear grammars), the associated \(q\)-derivative satisfies the \(q\)-Leibniz formula: 
\[
D^n(fg)=
\sum_{k=0}^{n}
{n \brack k}_q
D^k(f)\,
\uparrow^k\!\left(D^{n-k}(g)\right),
\]
where ${n \brack k}_q$, or simply ${n \brack k}$,
are the $q$-binomial coefficients.  
As shall be seen, the $q$-grammars are
associated with a sequence of noncommutative variables $x_0, x_1, x_2, \ldots$. 
For a given   \(q\)-grammar, 
they defined the shift operator
\[ \uparrow(x_i)=x_{i+1}, \]
which can be extended linearly and multiplicatively to expressions.

As an example, consider the $q$-Eulerian polynomials 
$A_n^{\inv}(q;x,y)$ defined by 
\[
A_n^{\inv}(q;x,y)
=
\sum_{\sigma\in S_n}
q^{\inv(\sigma)}
x^{\operatorname{asc}(\sigma)}
y^{\operatorname{des}(\sigma)},
\]
where $\inv(\sigma)$ denotes the number of inversions
of $\sigma$, that is,
\[
\inv(\sigma)
=
\lvert \{(i,j): 1\le i<j \le n,\ \sigma_i>\sigma_j\} \rvert.
\]
The first few values of 
$A_n^{\inv}(q;x,y)$ are:
\begin{align*}
A_1^{\inv}(q;x,y)
&=xy,\\[3pt]
A_2^{\inv}(q;x,y)&=x^2y+qxy^2,\\[3pt]
A_3^{\inv}(q;x,y)
&=x^3y+(2q+2q^2)x^2y^2+q^3xy^3.
\end{align*}

The grammar for the  Eulerian polynomials involves two
variables $x$ and $y$. For the $q$-grammar, 
we need two sequences of variables: $(x_0,x_1, x_2,\ldots)$ and $(y_0,y_1,y_2, \ldots)$. We further impose the following
order on the variables:
\begin{equation}\label{order}
x_0< y_0 < x_1 < y_1 < x_2 < y_2< \cdots .
\end{equation}
The $q$-grammar is given by
 \begin{equation}\label{q-inv}
 G_{\inv}=\{ x_j \rightarrow q^j y_j x_{j+1},\,    \;\; y_j \rightarrow q^j y_j x_{j+1} \;|\;
   j = 0,   1,   2,   \ldots \}.
\end{equation}
The \(q\)-derivative operator \(D\) associated with
\(G_{\inv}\) is defined by 
\[ D(w_1w_2\cdots w_n)= \sum_{i=1}^n\rho
\big( w_1 \cdots w_{i-1} R(w_i) \uparrow (w_{i+1} \cdots 
w_n) \big), \]
where we assume that $w_1w_2\cdots w_n$ is written
in the normal form, $R(w_i)$ means to 
apply the substitution rules 
to $w_i$, and $\rho$ denotes the normal form 
with respect to the order (\ref{order}). Set \(D^0(f)=f\), and for \(k\geq1\), define
\[ D^k(f) = D(D^{k-1}(f)).\]

Define the evaluation map
\[
\phi(x_j)=x,\qquad \phi(y_j)=y.
\]
Let $D$ denote the $q$-derivative  
with respect to the $q$-grammar \eqref{q-inv}. For $n\geq 1$, we have 
\[
\phi\bigl(D^n(x_0)\bigr)
=
A_n^{\inv}(q;x,y).
\]

For example,  
\begin{align*}
    D(x_0)& =R(x_0)=y_0x_1,\\[3pt]
    D^2(x_0) & = D(y_0x_1)=R(y_0)x_2+y_0R(x_1)=y_0x_1x_2+qy_0y_1x_2, \\[3pt] 
    D^3(x_0) & =D(y_0x_1x_2+qy_0y_1x_2) \\[3pt] 
    &=R(y_0)x_2x_3+y_0R(x_1)x_3+y_0x_1R(x_2)+qR(y_0)y_2x_3+qy_0R(y_1)x_3+qy_0y_1R(x_2)\\[3pt] 
    &=y_0x_1x_2x_3+qy_0y_1x_2x_3+q^2y_0x_1y_2x_3+qy_0x_1y_2x_3+q^2y_0y_1x_2x_3+q^3y_0y_1y_2x_3.
\end{align*}
Consequently, 
\begin{align*}
   \phi(D(x_0))& = xy,\\[3pt]
    \phi(D^2(x_0)) & =  x^2y+qxy^2, \\[3pt] 
    \phi(D^3(x_0)) & = yx^3+(2q+2q^2)y^2x^2+q^3y^3x.
\end{align*}

Using a grammatical labeling for \(G_{\inv}\),
Han-Ji-Xiong \cite{Han-Ji-Xiong2026} showed that
 the $q$-Eulerian polynomials 
 $A_n^{\inv}(q;x,y)$ can be generated by 
 the above $q$-grammar. Furthermore,
 they derived its \(q\)-exponential generating function by
 the $q$-grammatical calculus. They also found a \(q\)-grammar for the \(q\)-Roselle polynomials and derived their
\(q\)-exponential generating function. The 
 $q$-grammars for two kinds of the $q$-Andr\'e polynomials are also given.

\noindent {\bf Acknowledgments.}  
The author is grateful to the referees
for their valuable suggestions.

\end{document}